\documentclass[sn-mathphys]{sn-jnl}

\catcode`\|=12\relax  

\usepackage{amsmath}
\usepackage{amssymb}
\usepackage{amsthm}
\usepackage[shortlabels]{enumitem}
\usepackage{cleveref}  
\usepackage{xargs} 
\usepackage{aligned-overset}
\usepackage{anyfontsize}

\usepackage{tikz}
\usetikzlibrary{
    calc,
    patterns,
    angles,
    quotes,
    positioning,
    shapes.geometric
}

\definecolor{darkgreen}  {RGB}{ 72, 117,  73}
\definecolor{lightgreen} {RGB}{171, 210, 130}
\definecolor{lightgray}  {RGB}{167, 181, 183}
\definecolor{blue}       {RGB}{  3, 124, 135}
\definecolor{black}      {RGB}{ 16,  32,  32}
\definecolor{zibblue}    {RGB}{ 93, 188, 210}
\definecolor{firebrick}  {RGB}{178,  34,  34}
\definecolor{steelblue}  {RGB}{ 70, 130, 180}
\definecolor{darkblue}   {RGB}{ 42,  78, 108}
\definecolor{forestgreen}{RGB}{ 34, 139,  34}

\usepackage[colorinlistoftodos,prependcaption,textsize=small]{todonotes}
\newcommandx{\unsure}[2][1=]{\todo[linecolor=red,backgroundcolor=red!25,bordercolor=red,#1]{#2}}
\newcommandx{\change}[2][1=]{\todo[linecolor=blue,backgroundcolor=blue!25,bordercolor=blue,#1]{#2}}
\newcommandx{\info}[2][1=]{\todo[linecolor=OliveGreen,backgroundcolor=OliveGreen!25,bordercolor=OliveGreen,#1]{#2}}
\newcommandx{\improvement}[2][1=]{\todo[linecolor=Plum,backgroundcolor=Plum!25,bordercolor=Plum,#1]{#2}}
\newcommandx{\thiswillnotshow}[2][1=]{\todo[disable,#1]{#2}}

	\newtheorem{theorem}{Theorem}

	\newtheorem{lemma}[theorem]{Lemma}

	\newtheorem{definition}{Definition}

\renewcommand{\R}{\ensuremath{\mathbb{R}}}

\renewcommand{\L}{\ensuremath{\mathcal{L}}}
\renewcommand{\B}{\ensuremath{\mathcal{B}}}
\newcommand{\bigO}{\ensuremath{\mathcal{O}}}

\newcommand{\uB}{\ensuremath{\underline{\mathcal{B}}}}
\newcommand{\oB}{\ensuremath{\overline{\mathcal{B}}}}

\newcommand{\oG}{\ensuremath{\overline{\Gamma}}}

\newcommand{\ob}{\ensuremath{\overline{\beta}}}

\newcommand{\og}{\ensuremath{\overline{\gamma}}}

\newcommand{\uv}{\ensuremath{\underline{v}}}
\newcommand{\ov}{\ensuremath{\overline{v}}}
\newcommand{\oov}{\ensuremath{\overline{\overline{v}}}}

\newcommand{\oc}{\ensuremath{\overline{c}}}

\newcommand{\dL}{\ensuremath{\delta L}}

\newcommand{\ndL}{\ensuremath{|\delta L|}}

\newcommand{\dz}{\ensuremath{\delta z}}
\newcommand{\Dz}{\ensuremath{\Delta z}}

\newcommand{\dchi}{\ensuremath{\delta\chi}}
\newcommand{\Dchi}{\ensuremath{\Delta\chi}}

\newcommand{\dZ}{\ensuremath{\delta Z}}

\newcommand{\dl}{\ensuremath{\delta\lambda}}
\newcommand{\Dl}{\ensuremath{\Delta\lambda}}

\newcommand{\ndl}{\ensuremath{|\dl|}}

\newcommand{\dx}{\ensuremath{\delta\xi}}
\newcommand{\dxt}{\ensuremath{\delta\xi_\tau}}
\newcommand{\Dx}{\ensuremath{\Delta\xi}}
\newcommand{\Dxt}{\ensuremath{\Delta\xi_\tau}}

\newcommand{\xt}{\ensuremath{\xi_\tau}}
\newcommand{\nxt}{\ensuremath{\|\xt\|}}

\newcommand{\tx}{\ensuremath{\tilde{\xi}}}
\newcommand{\txt}{\ensuremath{\tilde{\xi}_\tau}}

\newcommand{\ndx}{\ensuremath{\|\dx\|}}
\newcommand{\ndxt}{\ensuremath{\|\dxt\|}}

\newcommand{\nDx}{\ensuremath{\|\Dx\|}}
\newcommand{\nDxt}{\ensuremath{\|\Dxt\|}}

\newcommand{\dX}{\ensuremath{\delta X}}

\newcommand{\glob}[1]{#1^{\star\star}}
\newcommand{\loc}[1]{#1^{\star}}

\newcommandx{\Cnorm}[3][ 1={0,1}, 2={0,1} ]{
	\ifthenelse{\equal{#1}{}}{
		\ifthenelse{\equal{#2}{}}{
			\ensuremath{\|#3\|_{C^{0,1}}}
		}
		{ 
			\ensuremath{\|#3\|_{C^{0,1}(#2)}}
		}
	}{
		\ifthenelse{\equal{#2}{}}{
			\ensuremath{\|#3\|_{C^{#1}}}
		}
		{ 
			\ensuremath{\|#3\|_{C^{#1}(#2)}}
		}
	}
}

\newcommandx{\Hnorm}[3][1={1}, 2={0,1} ]{
	\ifthenelse{\equal{#1}{}}{
		\ifthenelse{\equal{#2}{}}{
			\ensuremath{\|#3\|_{H^1}}
		}
		{ 
			\ensuremath{\|#3\|_{H^1(#2)}}
		}
	}{
		\ifthenelse{\equal{#2}{}}{
			\ensuremath{\|#3\|_{H^{#1}}}
		}
		{ 
			\ensuremath{\|#3\|_{H^{#1}(#2)}}
		}
	}
}

\newcommandx{\Lnorm}[3][ 1=2, 2={0,1} ]{
	\ifthenelse{\equal{#2}{}}{
		\ensuremath{\|#3\|_{L^{#1}}}
	}
	{ 
		\ensuremath{\|#3\|_{L^{#1}(#2)}}
	}
}

\newcommandx{\Lone}[2][1={0,1}]{
	\ifthenelse{\equal{#1}{}}{
		\ensuremath{\|#2\|_{L^{1}}}
	}
	{ 
		\ensuremath{\|#2\|_{L^{1}(#1)}}
	}
}

\newcommandx{\Ltwo}[2][1={0,1} ]{
	\ifthenelse{\equal{#1}{}}{
		\ensuremath{\|#2\|_{L^{2}}}
	}
	{ 
		\ensuremath{\|#2\|_{L^{2}(#1)}}
	}
}

\newcommandx{\Linf}[2][1={0,1} ]{
	\ifthenelse{\equal{#1}{}}{
		\ensuremath{\|#2\|_{L^{\infty}}}
	}
	{ 
		\ensuremath{\|#2\|_{L^{\infty}(#1)}}
	}
}

\newcommandx{\Ytwo}[1]{
    \ensuremath{\|#1\|_{Y^{2}}}
}

\newcommandx{\Yinf}[1]{
    \ensuremath{\|#1\|_{Y^{\infty}}}
}

\newcommandx{\Xinf}[1]{
    \ensuremath{\|#1\|_{X^{\infty}}}
}

\newcommandx{\Ztwo}[1]{
    \ensuremath{\|#1\|_{Z^{2}}}
}

\newcommandx{\Zinf}[1]{
    \ensuremath{\|#1\|_{Z^{\infty}}}
}

\newcommandx{\Nhood}[3][2={R}, 1={\glob x}]{
    \ensuremath{\mathcal{B}(#1, #2) #3}
}

\newcommand{\myindent}{\hspace{2cm}}

\DeclareMathAlphabet{\mathpzc}{OT1}{pzc}{m}{it}

\ifpdf
	\DeclareGraphicsExtensions{.eps,.pdf,.png,.jpg}
\else
	\DeclareGraphicsExtensions{.eps}
\fi

\newcommand{\orcid}[1]{Orcid: \href{#1}{https://orcid.org/#1}}

\jyear{2022}%

\begin{document}
\title
    {Newton's Method for Global Free Flight Trajectory Optimization}

\author[1]{\fnm{Ralf} \sur{Bornd\"orfer} (\orcid{0000-0001-7223-9174})} \email{borndoerfer@zib.de}
\equalcont{These authors contributed equally to this work.}

\author*[1,2]{\fnm{Fabian} \sur{Danecker} (\orcid{0000-0002-8953-808X})} \email{danecker@zib.de}
\equalcont{These authors contributed equally to this work.}

\author[2]{\fnm{Martin} \sur{Weiser} (\orcid{0000-0002-1071-0044})} \email{weiser@zib.de}
\equalcont{These authors contributed equally to this work.}

\affil[1]{
	\orgdiv{Network Optimization},
	\orgname{Zuse Institute Berlin},
	\orgaddress{
		\street{Takustr. 7},
		\postcode{14195},
		\city{Berlin},
		\country{Germany}%
	}%
}
\affil[2]{
	\orgdiv{Modeling and Simulation of Complex Processes},
	\orgname{Zuse Institute Berlin},
	\orgaddress{
		\street{Takustr. 7},
		\postcode{14195},
		\city{Berlin},
		\country{Germany}%
	}%
}

\abstract{
    Globally optimal free flight trajectory optimization can be achieved with a combination of discrete and continuous optimization. A key requirement is that Newton's method for continuous optimization converges in a sufficiently large neighborhood around a minimizer. We show in this paper that, under certain assumptions, this is the case.
}




\maketitle

\section{Introduction}

Around the world countries are implementing Free Flight airspaces that allow aircraft to choose their own route, as opposed to being restricted to a predetermined three-dimensional network. 
The primary factors that influence costs are time and fuel consumption, which are closely interrelated \cite{WellsEtAl2021}. Based on the relative weights of these factors (cf. cost index) the optimal airspeed can be determined, which typically remains largely constant \cite{KarischEtAl2012, AlizadehEtAl2018, RumlerEtAl2010}. Additionally, the vertical flight path can usually be predetermined using aircraft performance data \cite{NgEtAl14}. Consequently, the problem can be well approximated in a way proposed by Zermelo in 1931 \cite{Zermelo1931}, which involves finding the most efficient trajectory from point A to B for an aircraft flying at a constant airspeed in a given two-dimensional wind field.

The Free Flight Trajectory Optimization Problem is usually solved using direct or indirect methods from Optimal Control \cite{NgEtAl14, Betts11, DrevesEtAl2017, GeigerEtAl06, GirardetElAl2013, GirardetEtAl2014}. These are highly efficient, but suffer from one key drawback: They only converge locally and are thus dependent on a sufficiently good starting point. This makes such methods, used as a standalone, incapable of meeting airlines' high expectations regarding the global optimality of routes.

In \cite{BorndoerferDaneckerWeiser2021, BorndoerferDaneckerWeiser2022a, BorndoerferDaneckerWeiser2022b} a deterministic two-stage algorithm was proposed that combines discrete and continuous optimization in order to find a globally optimal solution to the free flight trajectory optimization problem. With this approach the exponential complexity of other branch and bound based algorithms is circumvented.

The primary objective of the first stage is to obtain a finite sample in a systematic manner that adequately covers the search space. This deterministic approach eliminates the potential for infinite runtime, which may occur when using stochastic global optimization algorithms, such as Particle Swarm Optimization, Simulated Annealing, or Monotonic Basin Hopping \cite{Locatelli2002,CassioliEtAl2013,AddisEtAl2011,BonyadiMichalewicz2017}.
\\
One approach is to create a locally dense directed graph with a specific density determined by the node spacing $h$ and connectivity length $\ell$, thereby implicitly defining the sample. The instances can then be selected in order of quality by applying Yen's algorithm \cite{Yen1971} to calculate the k\textsuperscript{th} shortest paths.
\\
Promising paths serve as initial guesses for a subsequent refinement stage in which a continuous solution to the problem is calculated up to the desired accuracy.

Analytical evidence and numerical experiments have demonstrated that the new hybrid algorithm has a time complexity of $\bigO(\ell^{-1})$, making it superior to the conventional purely discrete approach, which has a time complexity of $\bigO(\ell^{-6})$ \cite{BorndoerferDaneckerWeiser2021}. In this context, $\ell$ refers to the maximum arc length in a graph and the discretization length in a continuous optimization scenario. Thus, $\ell^{-1}$ serves as a comparable metric for the precision of the solution.

The present paper is concerned with the second stage. One way to generate a continuous solution is to apply Newton's method to the first order necessary conditions (the KKT-conditions) -- an approach commonly referred to as Newton-KKT or Sequential Quadratic Programming (SQP) (see e.g., \cite{NocedalWright2006}).
It is now shown that there is a quantifiable domain around a global optimum such that Newton-KKT converges if initialized accordingly.

Since the computational effort of the first graph-searching stage depends exclusively on the problem instance, i.e., the wind conditions, the algorithm asymptotically inherits the super fast convergence rates of the Newton-KKT method.

The paper is structured as follows. After defining the problem and introducing a formulation that is convenient for the analytical discussion in \Cref{sec:problem-formulation}, we formally state the necessary and sufficient conditions as well as the Newton-KKT approach in \Cref{sec:continuous-optimization}. The proof of convergence is provided in \Cref{sec:proof-of-convergence} followed by a conclusion emphasizing the impact on previous and future work.


\section{The Free Flight Trajectory Optimization Problem} \label{sec:problem-formulation}

\subsection{Notation}
Throughout this article lower case subscripts like e.g., $x_t$ or $\xt$, denote partial derivatives, while total derivatives are indicated by primes, e.g., $T'$ or $f'$.
Locally and globally optimal quantities are indicated by single and double superscript stars, respectively, e.g., $\loc\xi$ or $\glob\xi$.
If not stated otherwise, we assume $\|\cdot\|$ to denote the $l^2$-norm. Accordingly, we use the following quantitative definition of the $L^\infty$-norm in terms of the $l^2$-norm.
\begin{definition}
    Let $f: \Omega \to \R^n$. Then we define
    \begin{align}
        \Linf[\Omega] f := \inf\{ C \ge 0: \|f(x)\|_2 \le C ~\text{for a.a. } x\in\Omega \}.
    \end{align}
\end{definition}

\subsection{Problem Statement}

Neglecting any traffic flow restrictions, we consider Lipschitz-continuous flight paths $\xi\in C^{0,1}((0,1),\R^2)$ connecting origin $\xi(0)=x_O$ and destination $\xi(1) =x_D$. By Rademacher's theorem, such paths are almost everywhere differentiable, and moreover contained in the Sobolev space $W^{1,\infty}((0,1),\R^2)$.

A short calculation reveals that an aircraft travelling along such a path $\xi$ with constant airspeed $\ov$ through a three times continuously differentiable wind field $w\in C^3(\R^2,\R^2)$ with bounded magnitude $\|w(x)\| < \ov$ reaches the destination after a flight duration
\begin{equation}\label{eq:travel-time}
    T(\xi) = \int_0^1 f\big(\xi(\tau),\xt(\tau)\big)\, d\tau
\end{equation}
with $\xi_\tau$ denoting the time derivative of $\xi$ and
\begin{align} \label{eq:dt-dtau}
    f(\xi,\xt)
    := t_\tau
    = \frac{-\xt^Tw + \sqrt{(\xt^Tw)^2+(\ov^2 - w^Tw)(\xt^T \xt)}}{\ov^2 - w^Tw},
\end{align}
see~\cite{BorndoerferDaneckerWeiser2021, BorndoerferDaneckerWeiser2022a, BorndoerferDaneckerWeiser2022b}.

Among these paths $\xi$, we need to find one with minimal flight duration $T(\xi)$, since that is essentially proportional to fuel consumption~\cite{WellsEtAl2021}. This classic of optimal control is known as Zermelo's navigation problem~\cite{Zermelo1931}.
It can easily be shown that in case of bounded wind speed, the optimal trajectory cannot be arbitrarily longer than the straight connection of origin and destination. Hence, every global minimizer is contained in an ellipse $\Omega\subset\R^2$ with focal points $x_O$ and $x_D$.



The flight duration $T$ as defined in~\eqref{eq:travel-time} is based on a time reparametrization from actual flight time $t\in[0,T]$ to pseudo-time $\tau\in[0,1]$ according to the actual flight trajectory $x(t) = \xi(\tau(t))$ such that $\| x_t(t)-w(x(t))\| = \ov$, where $x_t$ denotes the so called ground speed, i.e., the derivative of position $x$ with respect to the unscaled time $t$.
As a consequence, the actual parametrization of $\xi$ in terms of pseudo-time $\tau$ is irrelevant for the value of $T$. Calling two paths $\xi,\tx$ equivalent if there exists a Lipschitz-continuous bijection $r:(0,1)\to(0,1)$ such that $\xi(r(\tau)) = \tx(\tau)$, we can restrict the optimization to equivalence classes.
Moreover, every equivalence class contains a representative with constant ground speed $\|\xt(\tau)\|=L$ for almost all $\tau$, that can be obtained from any $\tx$ with $\|\txt(\tau)\| \neq 0 ~ \forall\tau$ via
\begin{equation}
    \xi(\tau) := L\int_0^\tau \frac{\txt(t)}{\|\txt(t)\|} dt,
    \quad L:=\int_0^1 \|\txt(\tau)\| d\tau.
\end{equation}
Hence, we introduce $z:=(L,\xi) \in Z := \R\times X$ and the affine space of valid trajectories
\begin{equation}\label{eq:admissible-set}
    X := \{\xi\in W^{1,\infty}((0,1), \R^2) \; \mid \; \xi(0) = x_O, \; \xi(1) = x_D\}.
\end{equation}
and subsequently consider the equivalent constrained minimization problem
\begin{align}\label{eq:reduced-problem}
    \min_{z\in Z} T(\xi), \quad\text{s.t.} \quad h(z) = 0\quad \text{for a.a. $\tau\in(0,1)$}
\end{align}
with $h(z)=0$ expressing the constant ground speed requirement, as
\begin{align} \label{eq:constraint}
    h: Z \to \Lambda := L^2((0,1), \R), \quad  z \mapsto \xt^T\xt-L^2.
\end{align}
for $L\le L_{max}$ with an arbitrary continuation for $L>L_{max}$ that is linear in $\nxt$.
If the constraint is satisfied, $L$ also represents the path length, since
\begin{align} \label{eq:pathlength}
    \int_0^1 \nxt d\tau = L.
\end{align}
Note that $T:X\to\R$ is Fr\'echet differentiable with respect to the corresponding linear space
\begin{equation} \label{eq:definition-dX}
    \dX := W^{1,\infty}_0((0,1),\R^2)
\end{equation}
of directions $\dx$ with zero boundary values, which consequently do not change origin and destination. This space is equipped with the norm
\begin{align}
    \Xinf{\dx} = \Linf\dx + \Linf\dxt.
\end{align}
Further we define the linear space
\begin{align}
    \dZ := \R \times \dX
\end{align}
and equip the spaces $Z$ and $\dZ$ with the norms
\begin{subequations} \label{eq:z-norms}
\begin{alignat}{4}
    &\Zinf z &&= |L| &&+ \Linf\xi &&+ \Linf\xt \quad\text{and} \label{eq:zinf-norm}\\
    &\Ztwo z &&= |L| &&+ \Ltwo\xi &&+ \Ltwo\xt.  \label{eq:ztwo-norm}
\end{alignat}
\end{subequations}

\section{Continuous Optimization: Newton-KKT} \label{sec:continuous-optimization}

In order to find a continuous solution to the free flight optimization problem \eqref{eq:reduced-problem} we apply Newton's method to the first order necessary conditions (the KKT-conditions), which is also known as sequential quadratic programming (SQP).
Before we formally introduce Newton's method, we discuss the necessary and sufficient conditions for optimality, which also defines the goal of the presented algorithm.

\subsection{Optimality Conditions} \label{sec:optimality-conditions}

\subsubsection{Necessary Conditions}

The goal of the present paper is to find an isolated globally optimal solution $\glob\xi$ to~\eqref{eq:reduced-problem} that satisfies $T(\glob\xi) \le T(\xi) \; \forall\xi\in X$, contrary to a local optimizer $\loc\xi$ that is only superior to trajectories in a certain neighborhood, $T(\loc\xi) \le T(\xi) \; \forall\xi\in \mathcal{N}(\loc\xi) \subseteq X$.
An isolated global minimizer satisfies the necessary Karush-Kuhn-Tucker (KKT) optimality conditions~\cite{MaurerZowe1979} given that it is a regular point, which is always the case, as confirmed by the following Theorem.

\begin{theorem}
    Let $z=(L,\xi)\in Z$ with $L>0$ and assume there is a direction $u\in\R^2$ and $c>0$ such that $\xt^T u \ge c$ almost everywhere. Then, $h'(z):\dZ\to L^\infty(0,1)$ is surjective, i.e., $z$ is regular.
\end{theorem}
\begin{proof}
    Let $f\in L^\infty(0,1)$ be given and $b:=\xt^T u\ge c$. We set
    \[
        \dL = - \frac{\int_0^1 b^{-1}f/2\,d\tau}{L \int_0^1 b^{-1}\,d\tau}
    \]
    and
    \[
        g = b^{-1} \left( f/2 + L\dL\right), \quad
        \dxt = gu.
    \]
    Due to $b\ge c$ almost everywhere, $b^{-1}$ is bounded and hence $g,\xt \in L^\infty(0,1)$.  By construction, $\int_0^1 \dxt\,d\tau = 0$ holds, such that $\dz  = (\dL,\dx)\in\dZ$.
    Now we obtain
    \begin{align*}
        h'(z)[\dz]
        &= 2 \xt^T \dxt - 2L\dL \\
        &= 2bg - 2L\dL \\
        &= 2(f/2+L\dL) - 2L \dL \\
        &= f,
    \end{align*}
    and thus the claim.
\end{proof}

For $\lambda\in\Lambda^* = L^2((0,1), \R)$, the Lagrangian is defined as
\begin{equation} \label{eq:lagrangian}
	\L(z,\lambda) := T(\xi) + \langle \lambda, h(z)\rangle.
\end{equation}
The KKT-conditions guarantee for a regular minimizer $\glob{z}$ the existence of a Lagrange multiplier $\glob{\lambda}\in L^2(0,1)$, such that
\begin{alignat*}{3}
	0 & = \L_z(\glob{z},\glob{\lambda})[\dz] \qquad&& \forall~\dz && \in\dZ,   \\
	0 & = \langle\dl,h(\glob{z})\rangle            && \forall~\dl && \in L^2(0,1)
\end{alignat*}
hold, where $\dz:= (\dL,\dx) \in \dZ$.
In our case, these necessary conditions read
\begin{subequations} \label{eq:necessary-conditions}
\begin{align}
	0&=\underbrace{T'(\glob{\xi})[\dx]}_{=0 ~ \eqref{eq:necessary-conditions-unconstrained}}
		+ 2\int_0^1 \glob{\lambda} \left(\dxt^T \glob\xt - \dL\,\glob{L} \right) d\tau
		&&\forall~\dz\in\dZ,
		\label{eq:necessary-conditions-a}\\
	0&= \int_0^1 \dl \left((\glob\xt)^T \glob\xt - (\glob L)^2\right) \; d\tau &&\forall~ \dl\in  L^2(0,1).
		\label{eq:necessary-conditions-b}
\end{align}
\end{subequations}
Let us for a moment consider the unconstrained problem analogous to \eqref{eq:reduced-problem},
\begin{align} \label{eq:unconstrained-problem}
    \underset{\xi\in X}{\min}~T.
\end{align}
Any global minimizer $\glob{\tilde\xi}$ of~\eqref{eq:unconstrained-problem} is clearly non-isolated due to possible reparametrizations of the time. Let $\glob{\xi}$ denote the equivalent trajectory with constant ground speed, i.e., $\|\glob{\xt}(\tau)\|=\glob{L}$ for almost all $\tau$.
Both solutions $\glob{\tilde\xi},\glob{\xi}$ satisfy the first order necessary condition
\begin{align} \label{eq:necessary-conditions-unconstrained}
    0&=T'(\glob{\xi})[\dx] \quad \forall \dx\in\dX.
\end{align}
Moreover, $\glob{\xi}$ -- together with $\glob L$ from~\eqref{eq:pathlength} -- is a global minimizer of the constrained problem, which indicates that the ground-speed-constraint~\eqref{eq:constraint} is only weakly active. We confirm this by showing that the corresponding Lagrange multipliers $\glob{\lambda}$ vanish.

\begin{lemma} \label{th:lagrange-multipliers-vanish}
    Let $\glob z = (\glob{\xi},\glob{L})$ be a global minimizer of~\eqref{eq:reduced-problem}. Then, this solution together with
    \begin{equation} \label{eq:lagrange-multipliers-vanish}
        \glob{\lambda} = 0
    \end{equation}
    satisfies the necessary conditions~\eqref{eq:necessary-conditions}.
\end{lemma}

\begin{proof}
    Since $\glob{\xi}$ is also a global minimizer of the unconstrained problem, the necessary condition \eqref{eq:necessary-conditions-unconstrained} states that $T'(\glob{\xi})\dx=0$.
    The term $\int_0^1 \glob\lambda \left(\dxt^T \glob\xt - \dL\,\glob L\right) d\tau$ of~\eqref{eq:necessary-conditions-a} vanishes for $\glob\lambda=0$.
    \eqref{eq:necessary-conditions-b} is satisfied because $\|\glob\xt\| = \glob{L}$ for almost all $\tau\in(0,1)$.
\end{proof}

\subsubsection{Sufficient Conditions} \label{sec:sufficient-conditions}

Now we turn to the second order sufficient conditions for optimality. In general, a stationary point $(\loc{z},\loc{\lambda})$ is a strict minimizer, if, in addition to the necessary conditions above, the well known Ladyzhenskaya–Babuška–Brezzi (LBB) conditions (e.g.,~\cite{Braess2013}) are satisfied, which comprise a) the so called \emph{inf-sup} condition and b) the requirement that the Lagrangian's Hessian regarding~$z$, $\L_{zz}$, need be positive definite on the kernel of~$h'$.

The \emph{inf-sup} condition states that for the minimizer $\loc z$ there is a $\kappa > 0$ such that
\begin{align} \label{eq:inf-sup-condition}
    \inf_{\dl\ne 0\in L^2(0,1)} \sup_{\dz\in\dZ^2} \frac{\langle \dl, h'(\loc z)[\dz]\rangle}{\Ltwo\dl \Ztwo\dz} \ge \kappa.
\end{align}
Formally, the second part of the LBB-conditions requires that there is a $\uB > 0$ such that
\begin{align*}
    \L_{zz}(\loc{z})[\dz]^2 \ge \uB \; \Ztwo{\dz}^2
\end{align*}
for any $\dz\in\dZ$ that satisfies
\begin{align*}
    \langle \dl, h'(\loc{z})[\dz]\rangle = 0 \quad \forall~\dl\in L^2(0,1).
\end{align*}
In the present case, this reads
\begin{align}
    T''(\loc{\xi})[\dx]^2
        + 2\int_0^1 \loc{\lambda} (\dxt^T\dxt - \dL^2) d\tau
        \ge \uB \Ztwo{\dz}^2
\end{align}
for any $\dz \in \dZ$ such that
\begin{align*}
    \int_0^1 \dl \left(\dxt^T \loc\xt - \dL\,\loc{L} \right) d\tau=0 \quad \forall~\dl\in L^2(0,1).
\end{align*}
In case of a global minimizer $\glob z = (\glob\xi, \glob L)$, this can be reduced using $\glob\lambda=0$ from \Cref{th:lagrange-multipliers-vanish}. Moreover, the constraint is equivalent to requiring that $\dxt^T\glob\xt = \dL\,\glob{L}$ almost everywhere.
With this, we conclude that for any isolated global minimizer $\glob z$ of~\eqref{eq:reduced-problem} that satisfies the \emph{inf-sup} condition, there exists a $\uB>0$ such that
\begin{align} \label{eq:sufficient-conditions}
    T''(\glob{\xi})[\dx,\dxt]^2
        \ge \uB \Ztwo{\dz}^2
\end{align}
for any $\dz\in\dZ$ such that $\dxt^T\glob\xt = \dL\,\glob{L} $ almost everywhere.

It is important to note that the second order sufficient conditions are formulated in a $L^2$-setting, while differentiability only holds in $L^\infty$. This is known as \emph{two-norm-discrepancy} \cite{CasasTroeltzsch2015}.

\subsection{Newton's Method} \label{sec:newton}
In order to provide a more compact notation, we use $\chi = (z,\lambda) \in Z\times L^2(0,1) =: Y$ in this context and define $F$ as the total derivative of the Lagrangian,
\begin{align} \label{eq:newton-function}
    F: Z\times L^2(0,1) \to \dZ^* \times \Lambda =: Y^*, \qquad F(\chi) := \L'(z,\lambda).
\end{align}
On $Y$ we define the following norms,
\begin{subequations} \label{eq:ynorms}
\begin{alignat}{3}
    &\Yinf \chi &&= \Zinf z &&+ \Linf\lambda \quad\text{and} \label{eq:yinf-norm} \\
    &\Ytwo \chi &&= \Ztwo z &&+ \Ltwo\lambda. \label{eq:ytwo-norm}
\end{alignat}
\end{subequations}
The problem is now to find a $\glob\chi$ such that the first order necessary conditions for optimality as stated in \eqref{eq:necessary-conditions} are satisfied, which translates to
\begin{align} \label{eq:newton-problem-short}
    F(\glob\chi) = 0.
\end{align}
Applying Newton's method, we iteratively solve
\begin{align} \label{eq:saddle-point-problem-short}
    F'(\chi^k)[\Dchi^k] = -F(\chi^k)
\end{align}
for $\Dchi^k$ and proceed with $\chi^{k+1} \gets \chi^k + \Delta \chi^k$, starting with some initial value $\chi^0$.
In other words, in every iteration we need to find $(\Dz^k, \Dl^k)$ such that
\begin{subequations} \label{eq:saddle-point-problem}
\begin{align}
    T''(\xi^k)[\dx][\Dx^k]
    + \langle \lambda^k, h''(z^k)[\dz][\Dz^k] \rangle
    + \langle \Dl^k, h'(z^k)[\dz] \rangle \span\span\span \notag\\
        &= -T'(\xi^k)[\dx] - \langle \lambda^k, h'(z^k)[\dz] \rangle &&\forall\dz\in\dZ, \\
    \langle \dl, h'(z^k)[\Dz^k] \rangle &= -\langle \dl, h(z^k) \rangle &&\forall\dl\in L^2(0,1).
\end{align}
\end{subequations}


\section{Proof of Convergence} \label{sec:proof-of-convergence}

On the way to prove the existence of a non-empty domain $\Nhood[\glob\chi]{}$ such that Newton's method as defined in \Cref{sec:newton} converges to the corresponding global minimizer $\glob\chi$, if initialized with a starting point within this neighborhood, we first prove that the KKT-operator $F'$ is invertible and that the Newton step $\Dchi^k$ is always well defined.
Essentially, this is the case if the LBB-conditions as given in \eqref{eq:inf-sup-condition} and \eqref{eq:sufficient-conditions} are satisfied. Hence, we will show that there is a $R>0$ such that the \emph{inf-sup} condition is satisfied and that the Lagrangian is positive definite on the kernel of the constraints for any $\chi\in\Nhood[\glob\chi]{}$.
Further, we show that an affine covariant Lipschitz condition holds, which finally helps to complete the proof.

Before we get there, we recall the following Lemma from \cite[Lemma 7]{BorndoerferDaneckerWeiser2022a} which provides a bound for the path length of a global minimizer.

\begin{lemma} \label{th:L-bounds}
    Let $\glob{z} = (\glob{L},\glob{\xi})$ be a global minimizer of~\eqref{eq:reduced-problem}, let $\Linf[\Omega]{w} \le \oc_0$, and define $\tilde L = \|x_D-x_O\|$. Then it holds that
    \begin{align} \label{eq:L-bounds}
        \tilde L \le \glob{L} \le \frac{\ov+\oc_0}{\ov-\oc_0} \tilde L.
    \end{align}
\end{lemma}

As most of the subsequent results hold in a $L^\infty$-neighborhood of a minimizer, we introduce the following notation.

\begin{definition}
    We call the $L^\infty$-neighborhood of a point $z\in Z$ or $\chi\in Y$,
    \begin{subequations}
    \begin{align}
        \Nhood[z]{}    &:= \{\tilde z\in Z: \Zinf{\tilde z-z} \le R \} \quad\text{or} \\
        \Nhood[\chi]{} &:= \{\tilde \chi\in Y: \Yinf{\tilde\chi-\chi} \le R \},
    \end{align}
    \end{subequations}
    respectively.
\end{definition}
Moreover, we provide three simple yet useful bounds that hold in such a $L^\infty$-neighborhood of a minimizer.
\begin{lemma} \label{th:general-bounds}
    Let $\glob\chi = (\glob z, \glob\lambda)$ be a global minimizer of \eqref{eq:reduced-problem} and the corresponding Lagrange multipliers. Then for every $\chi\in\Nhood[\glob\chi]{}$ it holds that
    \begin{subequations} \label{eq:general-bounds}
    \begin{alignat}{2}
        \glob L - R &\le L &&\le \glob L + R, \\
        \glob L - R &\le \Linf\xt &&\le \glob L + R, \\
        0 &\le \Linf\lambda &&\le  R.
    \end{alignat}
    \end{subequations}
\end{lemma}
\begin{proof}
    The first two inequalities follow immediately, since a global minimizer satisfies the constraint from \eqref{eq:reduced-problem}. The latter one is a direct consequence of \Cref{th:lagrange-multipliers-vanish}.
\end{proof}

\subsection{Inf-Sup Condition}

We now show that the \emph{inf-sup} condition, introduced in \eqref{eq:inf-sup-condition}, holds in a certain neighborhood around a global minimizer.
First, however, we point out that deviations $\dx$ and $\dxt$ from a trajectory are inherently related and that the former is always bounded by the latter.

\begin{theorem}[Wirtinger's inequality] \label{th:integral-dx2-bound}
    Let $\dx \in  H^1_0(0,1)$. Then
    \begin{equation} \label{eq:integral-dx2-bound}
        \Ltwo{\dx}^2 \le \frac1\pi \Ltwo{\dxt}^2
    \end{equation}
    holds.
\end{theorem}

\begin{theorem} \label{th:inf-sup-condition}
    Let $\glob z$ be a global minimizer of \eqref{eq:reduced-problem}. Further, let there be a constant $c>0$ and some direction $u\in\R^2$ with $\|u\|=1$ such that $u^T \glob\xt \ge c$ for almost all $\tau\in (0,1)$.
    Then for any $z=(L,\xi)\in\Nhood[\glob z]{}$ with $R < c$ there is some $\kappa > 0$ such that
    \[
        \inf_{\lambda\ne 0\in L^2(0,1)}\sup_{\dz\in\dZ} \frac{\langle \lambda, h'(z)[\dz]\rangle}{ \Ltwo\lambda \Ztwo\dz} \ge \kappa
    \]
    with
    \[
        \kappa(R) = (c-R) \left[\frac14\left(1+\frac1\pi\right) + 2 \left(\frac{\ov + \oc_0}{\ov - \oc_0} + \frac{R}{\tilde L}\right)^2\right]^{-1/2}.
    \]
\end{theorem}
\begin{proof}
    For $f\in L^2(0,1)$ we define
    \[
        \overline f := \int_0^1 f \, d\tau \in \R \quad \text{and} \quad \tilde f = f - \overline{f},
    \]
    respectively, such that $(\overline f, \tilde f)_{L^2(0,1)} = 0$ and
    \[
        \Ltwo{f}^2 = \Ltwo{\tilde f + \overline f}^2
        = \Ltwo{\tilde f}^2 + \overline f^2.
    \]
    With
    \begin{align} \label{eq:b-bound}
        \frac{\ov + \oc_0}{\ov - \oc_0}\tilde L + R
        \underset{\eqref{eq:L-bounds}}{\ge} \glob L + R
        \ge b := \xi_\tau^T u
        \ge c - R
    \end{align}
    we choose $\dxt = \frac{1}{2}\tilde \lambda u$ and $\dL = \frac{1}{2L}\left(\overline{b\tilde\lambda}- (c-R) \overline{\lambda}\right)$. Note that $\dx\in\dX$ holds.
    For this choice, we obtain for $\dz=(\dL,\dx)$
    \begin{alignat*}{2}
        \langle \lambda, h'(z)[\dz]\rangle
        &=&& \int_0^1 (2\xt^T \dxt \lambda - 2 L \dL \lambda) \, d\tau \\
        &=&& \int_0^1 b\tilde\lambda \lambda \, d\tau - 2 L \dL \overline{\lambda} \\
        &=&& \int_0^1 (b\tilde\lambda^2 + b\tilde\lambda\overline{\lambda}) \,d\tau - 2 L \dL \overline{\lambda} \\
        \underset{\eqref{eq:b-bound}}&{\ge}&& (c-R) \, \Ltwo{\tilde\lambda}^2
            + \left(\int_0^1 b\tilde\lambda\,d\tau - 2 L \dL\right)\overline{\lambda} \\
        &=&& (c-R) \, \Ltwo{\tilde\lambda}^2
            + \left(\int_0^1 b\tilde\lambda\,d\tau - \overline{b\tilde\lambda} + (c-R) \overline\lambda \right)\overline{\lambda} \\
        &=&& (c-R) \left( \Ltwo{\tilde\lambda}^2 + \overline{\lambda}^2\right) \\
        &=&& (c-R) \, \Ltwo{\lambda}^2.
    \end{alignat*}
    Moreover, we have
    \[
        \Ltwo{\dxt} \le \frac{1}{2}  \Ltwo{\tilde\lambda}
    \]
    and, since clearly $c \le \glob L$,
    \begin{alignat*}{2}
        |\dL|
        &\le &&\frac{1}{2 L} \left( \Ltwo{b} \Ltwo{\tilde\lambda} + (c-R)  |\overline\lambda| \right) \\
        \underset{\eqref{eq:b-bound}}&{\le} &&\frac{1}{\tilde L} \left( (\glob L + R) \Ltwo{\tilde\lambda} + (c-R) |\overline\lambda| \right) \\
        &\le &&\left(\frac{\ov + \oc_0}{\ov - \oc_0} + \frac{R}{\tilde L}\right)  \left( \Ltwo{\tilde\lambda} + |\overline\lambda| \right),
    \end{alignat*}
    which implies
    \begin{alignat*}{2}
        \Ztwo{\dz}^2
        \underset{\eqref{eq:ztwo-norm}}&{=} 
            &&\Ltwo{\dx}^2 + \Ltwo{\dxt}^2 + \dL^2 \\
        \underset{\eqref{eq:integral-dx2-bound}}&{\le} 
            &&\left(1+\frac1\pi\right) \Ltwo{\dxt}^2 + \dL^2 \\
        &\le 
            &&\frac14\left(1+\frac1\pi\right) \Ltwo{\tilde\lambda}^2
            + \left(\frac{\ov + \oc_0}{\ov - \oc_0} + \frac{R}{\tilde L}\right)^2 \left(\Ltwo{\tilde\lambda} + \overline\lambda \right)^2 \\
        &\le 
            &&\frac14\left(1+\frac1\pi\right)\Ltwo{\tilde\lambda}^2
            + 2 \left(\frac{\ov + \oc_0}{\ov - \oc_0} + \frac{R}{\tilde L}\right)^2 \Ltwo{\tilde\lambda}^2
            + 2 \left(\frac{\ov + \oc_0}{\ov - \oc_0} + \frac{R}{\tilde L}\right)^2 \overline\lambda^2 \\
        &\le 
            &&\left[\frac14\left(1+\frac1\pi\right) + 2 \left(\frac{\ov + \oc_0}{\ov - \oc_0} + \frac{R}{\tilde L}\right)^2 \right] \left(
            \Ltwo{\tilde\lambda}^2
            + \overline\lambda^2
        \right)\\
        &= 
            &&\left[\frac14\left(1+\frac1\pi\right) + 2 \left(\frac{\ov + \oc_0}{\ov - \oc_0} + \frac{R}{\tilde L}\right)^2 \right] \Ltwo{\lambda}^2.
    \end{alignat*}
    Consequently,
    \[
        \langle \lambda, h'(z)[\dz]\rangle
        \ge (c-R) \left[\frac14\left(1+\frac1\pi\right) + 2 \left(\frac{\ov + \oc_0}{\ov - \oc_0} + \frac{R}{\tilde L}\right)^2\right]^{-1/2} \Ltwo\lambda \, \Ztwo\dz
    \]
    yields the claim.
\end{proof}

\subsection{Positive Definiteness of the Lagrangian}

The next step in order prove invertibility of the KKT-operator $F'(\chi)$, \eqref{eq:saddle-point-problem-short}, is to show that the second partial derivative of the Lagrangian $\L(\chi)$, \eqref{eq:lagrangian}, with respect to the state $z$ is positive definite on the kernel of the linearized constraints. On the way we derive a similar result for the objective  $T(\xi)$, \eqref{eq:travel-time} for which we first derive an upper bound for its third derivative.

\newcounter{lemma-dddf-upper-bound-stored}
\setcounter{lemma-dddf-upper-bound-stored}{\value{theorem}}
\renewcommand{\myindent}{\hspace{2.5cm}}
\begin{lemma} \label{th:dddf-upper-bound}
    Let $\Linf[\Omega]{w} \le \oc_0 \le \ov/\sqrt5$, $\Linf[\Omega]{w_x} \le \oc_1$, $\Linf[\Omega]{w_{xx}} \le \oc_2$, and $\Linf[\Omega]{w_{xxx}} \le \oc_3$ and define $\uv^2 := \ov^2 - \oc_0^2$. Then, for any $\xi\in X$, the third directional derivative of $f$ as given in~\eqref{eq:dt-dtau} is bounded by
    \begin{align}
        &\hspace{-2cm}|f'''(\xi,\xt) [\dx,\dxt]^2[\Dx,\Dxt]| \notag\\
        \le ~
        &\left(
            \og_0 \nxt \ndx^2
            + \og_2 \ndx \ndxt
            + \frac{\og_4}{\nxt} \ndxt^2
        \right)~ \nDx \notag\\
        + &\left(
            \og_1 \ndx^2
            + \frac{\og_3}{\nxt} \ndx \ndxt
            +\frac{\og_5}{\nxt^2} \ndxt^2
        \right)~ \nDxt
    \end{align}
    with $\og_i\ge0$, $i\in0,\dots,5$, given as
    \begin{align}
        \og_0 &= \frac{2}{\uv^4} \left(
            37 \oc_1^3
            +21 \oc_1 \oc_2 \uv
            +2 \oc_3 \uv^2
        \right),
        & \og_3 &= 40\frac{\oc_1}{\uv^2},
        \notag \\
        \og_1 &= \frac1{\uv^3} \left(29 \oc_1^2 + 7 \uv \oc_2\right),
        & \og_4 &= 20\frac{\oc_1}{\uv^2},
        \notag \\
        \og_2 &=\frac1{\uv^3} (57 \oc_1^2 + 13\uv \oc_2),
        & \og_5 &= 18 \frac1{\uv}.
    \end{align}
\end{lemma}
The proof can again be found in the appendix. With this result we can derive a bound for the third directional derivative of $T$.

\begin{theorem} \label{th:dddT-upper-bound}
    Let $(\glob{L}, \glob{\xi})$ be a global minimizer of~\eqref{eq:reduced-problem} and define $\tilde L := \|x_D-x_O\|$ and $\Dx := \xi-\glob{\xi}$. Moreover, let $\|w(p)\| \le \oc_0 \le \ov / \sqrt5$, $\|w_x(p)\| \le \oc_1$, $\|w_{xx}(p)\| \le \oc_2$, and $\|w_{xxx}(p)\| \le \oc_3$ for every $p\in\Omega$. Then, for any $\xi\in X$ with $\Xinf\Dx \le R < \tilde L$, it holds that
    \begin{equation} \label{eq:dddT-upper-bound}
        |T'''(\xi)[\dx]^2[\Dx]| \le \oG \left(\Ltwo\dx^2 + \Ltwo\dxt^2\right) \Cnorm\Dx \, .
    \end{equation}
    with $\Cnorm\Dx = \Linf\Dx + \Linf\Dxt$ and
    \begin{align} \label{eq:oG}
        \oG := 	\max
        \bigg\{
            &\left(\frac{\ov+\oc_0}{\ov-\oc_0} \tilde L + R\right) \og_0 + \frac{\og_2}{2}, \quad
            \frac{\og_4}{\tilde L - R} + \frac{\og_2}{2}, \notag\\
            &\og_1 + \frac{\og_3}{2(\tilde L - R)}, \quad
            \frac{\og_3}{2(\tilde L - R)} + \frac{\og_5}{(\tilde L - R)^2}
        \bigg\}
    \end{align}
    and $\og_0,\dots,\og_5$ as given in \Cref{th:dddf-upper-bound} above.
\end{theorem}
\begin{proof}
    From the definition of $T$ in \eqref{eq:travel-time}, we know that
    \begin{align*}
        T'''(\xi)[\dx]^2[\Dx]
        &=\int_0^1 \, f'''(\xi,\xt) [\dx,\dxt]^2[\Dx,\Dxt] d\tau.
    \end{align*}
    Inserting the bound from \Cref{th:dddf-upper-bound,th:general-bounds} above and using Young's inequality yields
    \begin{alignat*}{2}
        &&&|T'''(\xi)[\dx]^2[\Dx]|
        \\
        & \le &&\int_0^1 ~ \left(
            \og_0 \nxt \ndx^2
            + \og_2 \ndx \ndxt
            + \frac{\og_4}{\nxt} \ndxt^2
        \right)~ \nDx \\
        &&&\hspace{0.5cm}+ \left(
            \og_1 \ndx^2
            + \frac{\og_3}{\nxt} \ndx \ndxt
            + \frac{\og_5}{\nxt^2} \ndxt^2
        \right)~ \nDxt ~d\tau.
        \\
        & \le &&\Linf[]\Dx \int_0^1
            \og_0 \nxt \ndx^2
            + \og_2 \ndx \ndxt
            + \frac{\og_4}{\nxt} \ndxt^2
        d\tau \\
        &&& + \Linf[]\Dxt \int_0^1
            \og_1 \ndx^2
            + \frac{\og_3}{\nxt} \ndx \ndxt
            + \frac{\og_5}{\nxt^2} \ndxt^2
        d\tau
        \\
        \underset{\eqref{eq:general-bounds}}&{\le}
            &&\Linf[]\Dx \int_0^1
                \left(\frac{\ov+\oc_0}{\ov-\oc_0} \tilde L + R\right) \og_0 \ndx^2
                + \og_2 \ndx \ndxt
                + \frac{\og_4}{\tilde L - R} \ndxt^2
            d\tau \\
        &&& + \Linf[]\Dxt \int_0^1
            \og_1 \ndx^2
            + \frac{\og_3}{\tilde L - R} \ndx \ndxt
            + \frac{\og_5}{(\tilde L - R)^2} \ndxt^2
        d\tau
        \\
        \underset{\text{(Y)}}&{\le}
            &&\Linf[]\Dx \left[
                \left(
                    \left(\frac{\ov+\oc_0}{\ov-\oc_0} \tilde L
                    + R\right) \og_0 + \frac{\og_2}{2}
                \right) \Ltwo[]{\dx}^2
                + \left(
                    \frac{\og_4}{\tilde L - R}
                    + \frac{\og_2}{2}
                \right) \Ltwo[]{\dxt}^2
            \right]\\
            &&&+
            \Linf[]\Dxt \left[
                \left(\og_1 + \frac{\og_3}{2(\tilde L - R)} \right) \Ltwo[]{\dx}^2
                + \left(\frac{\og_3}{2(\tilde L - R)} + \frac{\og_5}{(\tilde L - R)^2}\right) \Ltwo[]{\dxt}^2
            \right]
        \\
        \underset{\eqref{eq:oG}}&{\le}
            &&\oG \left(\Ltwo\dx^2 + \Ltwo\dxt^2\right) ~ \Cnorm\Dx.
    \end{alignat*}
\end{proof}

Having bounded the third derivative of $T$, we can estimate the potential decay of $T''$ and thus derive a lower bound for the size of this neighborhood. Similarly, we can bound $h''$ and hence $\L_{zz}$.

\begin{theorem} \label{th:Lzz-lower-bound-neighborhood}
    Let $\Linf[\Omega]{w} \le \oc_0 < \ov/\sqrt5$, $\Linf[\Omega]{w_x} \le \oc_1$, $\Linf[\Omega]{w_{xx}} \le \oc_2$, and $\Linf[\Omega]{w_{xxx}} \le \oc_3$ and define $\tilde L := \|x_D-x_O\|$.
    Moreover, let $\glob\chi := (\glob z, \glob\lambda)$ be a globally optimal solution to problem~\eqref{eq:reduced-problem}, that satisfies the necessary and sufficient conditions~\eqref{eq:necessary-conditions}, \eqref{eq:inf-sup-condition}, and~\eqref{eq:sufficient-conditions} with $\uB>0$.
    Then there is a $0 < R < \min\left\{ \frac{\uB}{2\oG},\; \frac{\uB}{40},\; \frac{\tilde L}2 \right\}$ with $\oG$ from \Cref{th:dddT-upper-bound} such that
    \begin{align} \label{eq:ddT-lower-bound-neighborhood}
        \L_{zz}(\chi)[\dz]^2 \ge \frac\uB4 \Ztwo{\dz}^2
    \end{align}
    holds for any $\chi\in \Nhood[\glob\chi]{}$ and any $\dz\in \dZ$ such that $\xt^T \dxt = L\dL$ holds almost everywhere.
\end{theorem}

\begin{proof}
    Let $\Dx := \xi-\glob{\xi}$ and note that $\Linf\Dx \le
    \Zinf\Dz \le R < \frac{\uB}{2\oG}$. Then we obtain
    \begin{alignat*}{2}
        T''(\xi)[\dx,\dxt]^2
        &=
            &&T''(\glob{\xi})[\dx,\dxt]^2
            + \int_0^1 T'''(\xi+\nu\Dx)[\dx,\dxt]^2[\Dx,\Dxt]\,d\nu
        \\
        \underset{\eqref{eq:sufficient-conditions}}&{\ge}
            &&\uB \Ztwo{\dz}^2
            + \int_0^1 T'''(\xi+\nu\Dx)[\dx,\dxt]^2[\Dx,\Dxt]\,d\nu
        \\
        \underset{\eqref{eq:dddT-upper-bound}}&{\ge}
            &&\uB \Ztwo{\dz}^2
            - \oG (\Ltwo\dx^2 + \Ltwo\dxt^2)~ \Zinf{\Dz}
        \\
        \underset{\eqref{eq:ztwo-norm}}&{\ge} 
            &&\uB \Ztwo{\dz}^2
            - \oG \Ztwo{\dz}^2 ~ \Zinf\Dz,
        \\
        &\ge &&\frac\uB2 \Ztwo{\dz}^2.
    \end{alignat*}
    Further, we point out that
    \begin{align} \label{eq:R-bounded}
        R \le \frac{\tilde L}2 \le \frac{\glob L}2,
    \end{align}
    which together with the bounds from \Cref{th:general-bounds} yields
    \begin{alignat*}{2}
        \langle \lambda, h''(z)[\dz]^2 \rangle
        &= &&\int_0^1 \lambda \left( \dxt^T\dxt - \dL^2 \right) d\tau \\
        &= &&\int_0^1 \lambda \left( \ndxt^2 - \left(\frac{\xt^T\dxt}{L}\right)^2 \right) d\tau \\
        &\ge &&- \Linf\lambda \left(\Ltwo{\dxt}^2 + \int_0^1 \frac{\nxt^2 \ndxt^2}{L^2} \; d\tau \right) \\
        &\ge &&- \Linf\lambda \left(\Ltwo{\dxt}^2 + \frac{\Linf{\xt}^2}{L^2} \int_0^1 \ndxt^2 \; d\tau \right) \\
        \underset{\eqref{eq:general-bounds}}&{\ge}
            &&- R \left( \Ltwo{\dxt}^2 + \frac{(\glob L + R)^2}{(\glob L - R)^2} \Ltwo{\dxt}^2 \right) \\
        &\ge &&- R \left( 1 + \frac{(\glob L + R)^2}{(\glob L - R)^2} \right) \Ltwo{\dxt}^2 \\
        \underset{\eqref{eq:R-bounded}}&{\ge}
            &&- 10 R \Ltwo{\dxt}^2 \\
        &\ge &&- \frac{\uB}{4} \Ltwo{\dxt}^2 \\
        \underset{\eqref{eq:ztwo-norm}}&{\ge} &&- \frac{\uB}{4} \Ztwo{\dz}^2.
    \end{alignat*}
    Together, these bounds yield the claim with
    \begin{align*}
        \L_{zz}(\chi)[\dz]^2
        &= T''(\xi)[\dx]^2 + \langle \lambda, h''(z)[\dz]^2 \rangle \\
        &\ge \frac{\uB}{2} \Ztwo{\dz}^2 - \frac{\uB}{4} \Ztwo{\dz}^2 \\
        &\ge \frac{\uB}{4} \Ztwo{\dz}^2.
    \end{align*}
\end{proof}

\subsection{Upper Bound for the Lagrangian}

As a counterpart to the previous Lemma, we also derive an upper bound for $L_{zz}$ close to a minimizer.
Again we start with the underlying function $f$ in order to bound the error in the objective function $T$.

\newcounter{lemma-ddf-upper-bound-stored}
\setcounter{lemma-ddf-upper-bound-stored}{\value{theorem}}
\begin{lemma} \label{th:ddf2-upper-bound-mixed}
    Let $\Linf[\Omega]{w} \le \oc_0 \le \ov/\sqrt5$, $\Linf[\Omega]{w_x} \le \oc_1$, and $\Linf[\Omega]{w_{xx}} \le \oc_2$. Moreover, let $\uv^2 := \ov^2 - \oc_0^2$. Then, for any $\xi\in X$, the second directional derivative of $f$ as given in~\eqref{eq:dt-dtau} is bounded by
    \begin{align}
        |f''(\xi,\xt) [\dx,\dxt][\Dx,\Dxt]|
        \le
        \quad& \ob_0 \nxt \ndx \nDx \notag\\
        +& \ob_1 \left( \ndx \nDxt + \ndxt \nDx \right) \notag\\
        +& \ob_2 \nxt^{-1} \ndxt \nDxt
    \end{align}
    with
    \begin{align}
        \ob_0 = 14 \frac{\oc_1^2}{\uv^3} + 4 \frac{\oc_2}{\uv^2}, \qquad
        \ob_1 = 7\frac{\oc_1}{\uv^2}, ~ \text{and} \qquad
        \ob_2 = \frac4{\uv} .
    \end{align}
\end{lemma}
The proof can be found in the appendix.

\begin{theorem} \label{th:ddT-upper-bound}
    Let $\glob z = (\glob{L},\glob{\xi})$ be a global minimizer of~\eqref{eq:reduced-problem} and $\Dz := z - \glob z$. Moreover, let $\Linf[\Omega]{w} \le \oc_0 \le \bar v / \sqrt{5}$, $\Linf[\Omega]{w_x} \le \oc_1$, and $\Linf[\Omega]{w_{xx}} \le \oc_2$. Also define $\underbar{v}^2 := \ov^2 - \oc_0^2$ and $\tilde L := \|x_D-x_O\|$.
    Then, for any $z\in \Nhood[\glob z]{}$, the second directional derivative of $T$ as defined in~\eqref{eq:travel-time} is bounded by
    \begin{align} \label{eq:ddT-upper-bound}
        |T''(\xi)[\Dx]^2|
        &\le\oB \Ztwo{\Dz}^2
    \end{align}
    with $\oB := \ob_1 + \max\left\{
        \left(\frac{\ov+\oc_0}{\ov-\oc_0} \tilde L + R\right) \ob_0, \; \frac{\ob_2}{\tilde L - R}
    \right\}$ and $\ob_0,\ob_1,\ob_2$ as defined in \Cref{th:ddf2-upper-bound-mixed}.
\end{theorem}
\begin{proof}
    From the definition of $T$ in~\eqref{eq:travel-time} we know that
    \begin{align*}
        T''(\xi)[\Dx,\Dxt]^2 = \int_0^1 f''[\Dx,\Dxt]^2 d\tau,
    \end{align*}
    which, together with the bounds from \Cref{th:ddf2-upper-bound-mixed,,th:general-bounds} as well as Young's inequality, then leads to
    \begin{alignat*}{2}
        |T''(\xi)[\Dx,\Dxt]^2|
        &\le
        &&\int_0^1 \left(
            \ob_0 \nxt \nDx^2
            + 2\ob_1 \nDx \nDxt
            + \frac{\ob_2}{\nxt} \nDxt^2
        \right) \, d\tau
        \\
        \underset{\eqref{eq:general-bounds}}&{\le}
            &&\ob_0 (\glob L + R) \int_0^1 \nDx^2 d\tau \\
            &&&+ 2\ob_1 \int_0^1 \nDx \nDxt d\tau \\
            &&&+ \frac{\ob_2}{\glob L - R} \int_0^1 \nDxt^2 d\tau
        \\
        \underset{\text{(Y)}}&{\le}
            &&\left( (\glob L + R) \ob_0 + \ob_1 \right) \Ltwo{\Dx}^2 \\
            &&&+ \left(\ob_1 + \frac{\ob_2}{\glob L - R} \right) \Ltwo{\Dxt}^2
        \\
        \underset{\eqref{eq:L-bounds}}&{\le}
            &&\left( \left(\frac{\ov+\oc_0}{\ov-\oc_0} \tilde L + R\right) \ob_0 + \ob_1 \right) \Ltwo{\Dx}^2 \\
            &&&+ \left(\ob_1 + \frac{\ob_2}{\tilde L - R} \right) \Ltwo{\Dxt}^2
        \\
        &\le &&\oB \left( \Ltwo{\Dx}^2 + \Ltwo{\Dxt}^2 \right)
        \\
        \underset{\eqref{eq:ztwo-norm}}&{\le} &&\oB \Ztwo{\Dz}^2.
    \end{alignat*}
\end{proof}

\begin{theorem} \label{th:Lzz-upper-bound}
    Let $\glob\chi = (\glob z, \glob\lambda)$ be a global minimizer of \eqref{eq:reduced-problem} and the corresponding Lagrange multipliers. Then for every $\chi\in\Nhood[\glob\chi]{}$ and every $\dz\in\dZ$ it holds that
    \begin{align} \label{eq:Lzz-upper-bound}
        |\L_{zz}(\chi)[\dz]^2| \le \left( \oB + R \right) \Ztwo{\dz}^2
    \end{align}
    with $\oB(R)$ from \Cref{th:ddT-upper-bound}.
\end{theorem}
\begin{proof}
    Using the bound from \Cref{th:ddT-upper-bound} and Young's inequality, we get
    \begin{alignat*}{2}
        |\L_{zz}(\chi)[\dz]^2|
            &= &&|T''(\xi)[\dx]^2 + \langle \lambda, h''(z)[\dz]^2 \rangle| \\
            \underset{\eqref{eq:ddT-upper-bound}}&{\le}
                &&\oB \Ztwo{dz}^2
                + \int_0^1 |\lambda \left( \dxt^T\dxt - \dL^2 \right)| \; d\tau \\
            &\le &&\oB \Ztwo{dz}^2
                + \Linf\lambda \left( \Ltwo{\dxt}^2 + \dL^2 \right) \\
            \underset{\eqref{eq:general-bounds}}&{\le}
                &&\oB \Ztwo{dz}^2
                + R \left( \Ltwo{\dxt}^2 + \dL^2 \right) \\
            \underset{\eqref{eq:ztwo-norm}}&{\le}
                &&\left( \oB + R \right) \Ztwo{dz}^2 .
    \end{alignat*}
\end{proof}

\subsection{Invertibility of the KKT-Operator}

Using the previous three results, which together state the existence of a neighborhood around a minimizer such that the LBB-conditions are satisfied, we are now ready to prove that the KKT-operator $F'$ is invertible.

\begin{lemma} \label{th:Finv-upper-bound}
    Let $\glob\chi = (\glob z, \glob\lambda)$ be a global minimizer of \eqref{eq:reduced-problem}, that satisfies the first and second order conditions for optimality with some $\uB > 0$, and the corresponding Lagrange multipliers.
    Further, let there be a $u$ with $\|u\|=1$ such that $u^T \glob\xt \ge c > 0$ for almost all $\tau\in(0,1)$.
    Then for $F$ as given in \eqref{eq:newton-function} it holds that
    \begin{align}
        \|F'(\chi)^{-1}\|_{(Y^2)^* \to Y^2} \le \omega_1
    \end{align}
    for every $\chi=(z,\lambda)\in\Nhood[\glob\chi]{}$ and
    \begin{align}
        \omega_1 = \sqrt2 \max\left\{
            \frac4\uB, \;
            \frac1\kappa \left(1+\frac{4(\oB+R)}{\uB}\right), \;
            \frac{\oB+R}{\kappa^2}
        \right\}
    \end{align}
    and $\oB(R)$ and $\kappa(R)$ as given in \Cref{th:ddT-upper-bound} and \Cref{th:inf-sup-condition}, respectively.
\end{lemma}
\begin{proof}
    The proof builds on some prerequisites that have been established above and are briefly summarized.
    \begin{enumerate}[label=\roman*)]
        \item In \Cref{th:inf-sup-condition} it was proved that the \emph{inf-sup} condition is satisfied:
            \[
                \underset{\lambda\in L^2(0,1)}{\inf} \; \underset{\dz\in\dZ}{\sup} \; \frac{\langle \lambda, h'(z)[\dz]\rangle}{\Ztwo\dz \Ltwo\lambda} \ge \kappa > 0.
            \]

        \item In \Cref{th:Lzz-lower-bound-neighborhood} it was proved that $\L_{zz}$ is positive definite on the kernel of the constraints, i.e.,
            \[
                \L_{zz}(\chi)[\dz]^2 = T''(\xi)[\dx]^2 + \langle \lambda, h''(z)[\dz]^2 \rangle \ge \frac\uB4 \Ztwo{\dz}^2
            \]
        for all $\dz\in\dZ$ such that $h'(z)[\dz] = 0$.
        \item In \Cref{th:Lzz-upper-bound} it was proved that $\L_{zz}$ is bounded from above as
            \[
                |\L_{zz}(\chi)[\dz]^2|
                = |T''(\xi)[\dx]^2 + \langle \lambda, h''(z)[\dz]^2 \rangle|
                \le (\oB + R) \Ztwo{\dz}^2 .
            \]
    \end{enumerate}
    Under these conditions, it follows from \emph{Brezzi's Splitting Theorem} \cite[Thm.~4.3]{Braess2013} that $F'(x)$ is isomorphic.
    Further, it can be shown that for every right hand side $F(x)$ of the saddle point problem \eqref{eq:saddle-point-problem-short} there is exactly one solution $(\Dz, \Dl)$ with
    \begin{align*}
        \Ztwo\Dz
            &\le \frac4\uB \; \Ztwo{T'(\xi) + \langle \lambda, h'(z) \rangle} \\
            &\quad+ \frac1\kappa \left(1+\frac{4(\oB+R)}{\uB}\right) \Ltwo{h(z)}, \\
        \Ltwo\Dl
            &\le \frac1\kappa \left(1+\frac{4(\oB+R)}{\uB}\right) \Ztwo{T'(\xi) + \langle \lambda, h'(z) \rangle} \\
            &\quad+ \frac{\oB+R}{\kappa^2} \left(1+\frac{4(\oB+R)}{\uB}\right) \Ltwo{h(z)}.
    \end{align*}
    With $\|F(\chi)\|_{(Y^2)^*} = \Ztwo{T'(\xi) + \langle \lambda, h'(z) \rangle}^2 + \Ltwo{h(z)}^2$ follows that
    \begin{align*}
        \Ztwo{\Dz} &\le \sqrt2 \max\left\{
            \frac4\uB, \;
            \frac1\kappa \left(1+\frac{4(\oB+R)}{\uB}\right)
        \right\} \|F(\chi)\|, \\
        \Ltwo\Dl &\le \sqrt2 \max\left\{
            \frac1\kappa \left(1+\frac{4(\oB+R)}{\uB}\right), \;
            \frac{\oB+R}{\kappa^2}
        \right\} \|F(\chi)\|,
    \end{align*}
    which directly yields
    \begin{align*}
        \Ytwo{\Dchi}^2 \underset{\eqref{eq:ytwo-norm}}{=} \Ztwo{\Dz}^2 + \Ltwo{\Dl}^2 \le \omega_1^2 \|F(\chi)\|
    \end{align*}
    with $\omega_1 = \sqrt2 \max\left\{
            \frac4\uB,
            \frac1\kappa \left(1+\frac{4(\oB+R)}{\uB}\right),
            \frac{\oB+R}{\kappa^2}
        \right\}$.
    This completes the proof, since
    \begin{align*}
        \|F'(\chi)^{-1}\|_{(Y^2)^* \to Y^2}
        = \underset{\|F(\chi)\|_{(Y^*)^2}}{\sup} \frac{\Ytwo{\Dchi}}{\|F(\chi)\|_{(Y^*)^2}}
        \le \omega_1.
    \end{align*}
\end{proof}

\subsection{Lipschitz Constant}

We are on the verge of presenting a Lipschitz constant for the free flight problem. To accomplish this, we introduce an additional bound in the form of a Lemma. This bound incorporates the constant $\hat\B$, which is derived in the appendix (\Cref{th:delta-f-bound}). It serves to define an upper limit on the second derivative of $f$ as defined in \Cref{eq:dt-dtau}. Its value is contingent upon the overall characteristics of the wind field.

\begin{lemma} \label{th:dF2-dF1-upper-bound}
    Let $\glob\chi = (\glob z, \glob\lambda)$ be a global minimizer of \eqref{eq:reduced-problem} and the corresponding Lagrange multipliers.
    For any $\chi_{i\in\{1,2\}} \in \Nhood[\glob\chi]{}$ there is a $\hat \B$ such that
    \begin{align}
        \|(F'(\chi_2) - F'(\chi_1)) [\chi_2-\chi_1]\|_{(Y^2)^*} \le \omega_2 \Ytwo{\chi_2-\chi_1}
    \end{align}
    with
    \begin{align}
        \omega_2 = (8+\hat\B) R.
    \end{align}
\end{lemma}

\begin{proof}
    From \Cref{th:general-bounds} it directly follows that
    \begin{subequations} \label{eq:bound-inf-norm}
    \begin{align}
        |L_2-L_1|                        &\le 2R, \\
        \Linf{\xi_{\tau,2}-\xi_{\tau,1}} &\le 2R, \\
        \Linf{\lambda_2-\lambda_1}       &\le R.
    \end{align}
    \end{subequations}
    Using these bounds as well as the Cauchy-Schwarz inequality and Young's inequality, we show that for any $\dchi \in \dZ \times L^2(0,1)$ with 
    it holds that
    \begin{alignat*}{2}
        &&&\hspace{-1cm} | \langle \lambda_2, h''(z_2)[z_2-z_1, \dz]\rangle - \langle \lambda_1, h''(z_1)[z_2-z_1, \dz]\rangle | \\
        &=&& |\int_0^1 \lambda_2 (\dxt^T (\xi_{\tau,2}-\xi_{\tau,1}) - \dL (L_2-L_1)) \\
            &&& - \lambda_1 (\dxt^T (\xi_{\tau,2}-\xi_{\tau,1}) - \dL (L_2-L_1)) d\tau | \\
        &=&& |\int_0^1 (\lambda_2- \lambda_1) (\dxt^T (\xi_{\tau,2}-\xi_{\tau,1}) - \dL (L_2-L_1)) d\tau | \\
        &\le&& \int_0^1 |\lambda_2- \lambda_1| \ndxt \|\xi_{\tau,2}-\xi_{\tau,1}\| d\tau \\
            &&&+ \ndL |L_2-L_1| \int_0^1 |\lambda_2-\lambda_1| d\tau \\
        \underset{\text{(CS)}}&{\le}&& 
            \left[ \int_0^1 \ndxt^2 d\tau \right]^{1/2}
            \left[ \int_0^1 (\lambda_2- \lambda_1)^2 \|\xi_{\tau,2}-\xi_{\tau,1}\|^2 d\tau \right]^{1/2} \\
            &&& + \ndL \; |L_2-L_1| \Lone[]{\lambda_2-\lambda_1} \\
        \underset{\eqref{eq:bound-inf-norm}}&{\le}&&
            \Ltwo[]\dxt
            \left[ 2R^2 \int_0^1 |\lambda_2-\lambda_1| \|\xi_{\tau,2}-\xi_{\tau,1}\| d\tau \right]^{1/2}\\
            &&& + R \; \ndL \; |L_2-L_1| \\
        \underset{\text{(CS)}}&{\le}&&
            \sqrt2 R \Ltwo[]\dxt
            \Ltwo[]{\lambda_2-\lambda_1}^{1/2}
            \Ltwo[]{\xi_{\tau,2}-\xi_{\tau,1}}^{1/2}\\
            &&& + R \; \ndL \; |L_2-L_1| \\
        \underset{\text{(Y)}}&{\le}&&
            \frac{\sqrt2}2 R \Ltwo[]\dxt \left[
                \Ltwo[]{\lambda_2-\lambda_1}
                + \Ltwo[]{\xi_{\tau,2}-\xi_{\tau,1}}
            \right]\\
            &&& + R \; \ndL \; |L_2-L_1| \\
        \underset{\eqref{eq:ytwo-norm}}&{\le}&&
            R \Ytwo{\dchi} \left[
                \Ltwo[]{\lambda_2-\lambda_1}
                + \Ltwo[]{\xi_{\tau,2}-\xi_{\tau,1}}
                + |L_2-L_1|
            \right] \\
        &\le&&
            R \Ytwo{\dchi} \left[
                \Ltwo[]{\lambda_2-\lambda_1}^2
                + \Ltwo[]{\xi_2-\xi_1}^2
                + \Ltwo[]{\xi_{\tau,2}-\xi_{\tau,1}}^2
                + |L_2-L_1|^2
            \right]^{1/2} \\
        \underset{\eqref{eq:ytwo-norm}}&{=}&&
            2 R \Ytwo{\dchi} \Ytwo{\chi_2-\chi_1}
    \end{alignat*}
    as well as
    \begin{alignat*}{2}
        &&&\hspace{-1cm} |\langle \lambda_2-\lambda_1, (h'(z_2)-h'(z_1))[\dz]\rangle| \\
        &=&& | \int_0^1 (\lambda_2-\lambda_1) \left((\xi_{\tau,2}-\xi_{\tau,1})^T \dxt - (L_2-L_1)\dL \right) d\tau | \\
        &\le&& \int_0^1 |\lambda_2-\lambda_1| \|\xi_{\tau,2}-\xi_{\tau,1}\| \ndxt d\tau \\
            &&&+ |L_2-L_1| \ndL  \int_0^1 |\lambda_2-\lambda_1| d\tau \\
        \underset{\text{(CS)}}&{\le}&&
            \left[\int_0^1 \ndxt^2 d\tau \right]^{1/2}
            \left[\int_0^1 (\lambda_2-\lambda_1)^2 \|\xi_{\tau,2}-\xi_{\tau,1}\|^2 d\tau\right]^{1/2} \\
            &&&+ |L_2-L_1| \ndL \Lone[]{\lambda_2-\lambda_1} \\
        \underset{\eqref{eq:bound-inf-norm}}&{\le}&&
            \Ltwo[]\dxt
            \left[2R^2 \int_0^1 |\lambda_2-\lambda_1| \; \|\xi_{\tau,2}-\xi_{\tau,1}\| d\tau\right]^{1/2} \\
            &&&+ R |L_2-L_1| \ndL \\
        \underset{\eqref{eq:ytwo-norm}}&{\le}&&
            \sqrt2 R \Ytwo{\dchi} 
            \left[\int_0^1 |\lambda_2-\lambda_1| \; \|\xi_{\tau,2}-\xi_{\tau,1}\| d\tau\right]^{1/2} \\
            &&&+ R \Ytwo{\dchi} |L_2-L_1| \\
        \underset{\text{(CS)}}&{\le}&&
            \sqrt2 R \Ytwo{\dchi} \Ltwo[]{\lambda_2-\lambda_1}^{1/2} \Ltwo[]{\xi_{\tau,2}-\xi_{\tau,1}}^{1/2} \\
            &&&+ R \Ytwo{\dchi} |L_2-L_1| \\
        \underset{\text{(Y)}}&{\le}&&
            \frac{\sqrt2}2 R \Ytwo{\dchi} \left[ \Ltwo[]{\lambda_2-\lambda_1} + \Ltwo[]{\xi_{\tau,2}-\xi_{\tau,1}} \right] \\
            &&&+ R \Ytwo{\dchi} |L_2-L_1| \\
        &\le&&
            R \Ytwo{\dchi} \bigg[ \Ltwo[]{\lambda_2-\lambda_1} + \Ltwo[]{\xi_{\tau,2}-\xi_{\tau,1}}
                + \Ltwo[]{\xi_2-\xi_1} + |L_2-L_1|\bigg] \\
        &\le&&
            2 R \Ytwo{\dchi} \bigg[
                \Ltwo[]{\lambda_2-\lambda_1}^2
                + \Ltwo[]{\xi_{\tau,2}-\xi_{\tau,1}}^2
                + \Ltwo[]{\xi_2-\xi_1}^2
                + |L_2-L_1|^2
            \bigg]^{1/2} \\
        \underset{\eqref{eq:ytwo-norm}}&{=}&&
            2 R \Ytwo{\dchi} \Ytwo{\chi_2-\chi_1}
    \end{alignat*}
    and
    \begin{alignat*}{2}
        &&&\hspace{-1cm} |\langle \dl, (h'(z_2)-h'(z_1))[z_2-z_1]\rangle| \\
        &=&& | \int_0^1 \dl ((\xi_{\tau,2}-\xi_{\tau,1})^T (\xi_{\tau,2}-\xi_{\tau,1}) - (L_2-L_1)^2) d\tau | \\
        &\le&& \int_0^1 \ndl \|\xi_{\tau,2}-\xi_{\tau,1}\|^2 d\tau
        + (L_2-L_1)^2 \int_0^1 \ndl d\tau \\
        \underset{\eqref{eq:bound-inf-norm}}&{\le}&&
            2R \int_0^1 \ndl \|\xi_{\tau,2}-\xi_{\tau,1}\| d\tau
            + 2R |L_2-L_1| \Lone[]\dl \\
        \underset{\text{(CS)}}&{\le}&&
            2R \left[\int_0^1 \dl^2 d\tau\right]^{1/2} \left[\int_0^1 \|\xi_{\tau,2}-\xi_{\tau,1}\|^2 d\tau\right]^{1/2} \\
            &&&+ 2R |L_2-L_1| \; \Lone[]\dl \\
        &\le&&
            2R \Ltwo[]\dl \Ltwo[]{\xi_{\tau,2}-\xi_{\tau,1}} \\
            &&&+ 2R |L_2-L_1| \; \Lone[]\dl \\
        \underset{\eqref{eq:ytwo-norm}}&{\le}&&
            2R \Ytwo{\dchi} \left[\Ltwo[]{\xi_{\tau,2}-\xi_{\tau,1}} + |L_2-L_1| \right] \\
        &\le&&
            4 R \Ytwo{\dchi} \bigg[
                \Ltwo[]{\lambda_2-\lambda_1}^2
                + \Ltwo[]{\xi_{\tau,2}-\xi_{\tau,1}}^2 \\
                &&&\hspace{1.6cm}+ \Ltwo[]{\xi_2-\xi_1}^2
                + |L_2-L_1|^2
            \bigg]^{1/2} \\
        \underset{\eqref{eq:ytwo-norm}}&{=}&&
            4 R \Ytwo{\dchi} \Ytwo{\chi_2-\chi_1}.
    \end{alignat*}
    As shown in \Cref{th:delta-f-bound} in the appendix, there is a $\hat\B<\infty$ such that
    \begin{multline*}
        |\left(f''(\xi_2, \xi_{\tau,2}) - f''(\xi_1, \xi_{\tau,1})\right)[\xi_2-\xi_1, \xi_{\tau,2}-\xi_{\tau,1}][\dx,\dxt] | \\
        \le \hat\B R \sqrt{\|\xi_2-\xi_1\|^2 + \|\xi_{\tau,2}-\xi_{\tau,1}\|^2} \sqrt{\ndx^2 + \ndxt^2},
    \end{multline*}
    which provides the following bound, as
    \begin{alignat*}{2}
        &&&\hspace{-1cm} |\left(T''(\xi_2) - T''(\xi_1)\right)[\xi_2-\xi_1, \dx]| \\
        &=&& 
            |\int_0^1 \left(f''(\xi_2, \xi_{\tau,2}) - f''(\xi_1, \xi_{\tau,1})\right)[\xi_2-\xi_1, \xi_{\tau,2}-\xi_{\tau,1}][\dx,\dxt] d\tau | \\
        &\le&& 
            \hat\B R \int_0^1 \left[\|\xi_2-\xi_1\|^2 + \|\xi_{\tau,2}-\xi_{\tau,1}\|^2\right]^{1/2} \left[\ndx^2 + \ndxt^2\right]^{1/2} d\tau \\
        \underset{\text{(CS)}}&{\le}&&
            \hat\B R \left[\int_0^1 \|\xi_2-\xi_1\|^2 + \|\xi_{\tau,2}-\xi_{\tau,1}\|^2 d\tau \right]^{1/2}
            \left[\int_0^1 \ndx^2 + \ndxt^2 d\tau\right]^{1/2} \\
        &\le&& 
            \hat\B R \left[\Ltwo[]{\xi_2-\xi_1}^2 + \Ltwo[]{\xi_{\tau,2}-\xi_{\tau,1}}^2\right]^{1/2}
            \left[ \Ltwo[]{\dx}^2 + \Ltwo[]{\dxt}^2\right]^{1/2} \\
        \underset{\eqref{eq:ytwo-norm}}&{\le}&&
            \hat\B R \Ytwo{\chi_2-\chi_1} \Ytwo{\dchi}.
    \end{alignat*}
    Finally, we use the bounds derived above to show that for any $\dchi$ it holds that
    \begin{align*}
        |(F'(\chi_2) - F'(\chi_1)) [\chi_2-\chi_1, \dchi]|
        &= | \left(T''(\xi_2) - T''(\xi_1)\right)[\xi_2-\xi_1, \dx] \\
            &\qquad + \langle \lambda_2, h''(z_2)[z_2-z_1, \dz]\rangle  \\
            &\qquad - \langle \lambda_1, h''(z_1)[z_2-z_1, \dz]\rangle \\
            &\qquad + \langle \lambda_2-\lambda_1, (h'(z_2)-h'(z_1))[\dz]\rangle \\
            &\qquad + \langle \dl, (h'(z_2)-h'(z_1))[z_2-z_1]\rangle |
        \\
        &\le \hat\B R \Ytwo{\dchi} \Ytwo{\chi_2-\chi_1} \\
            &\quad+ 2R \Ytwo{\dchi} \Ytwo{\chi_2-\chi_1} \\
            &\quad+ 2R \Ytwo{\dchi} \Ytwo{\chi_2-\chi_1} \\
            &\quad+ 4R \Ytwo{\dchi} \Ytwo{\chi_2-\chi_1}
        \\
        &= \omega_2 \Ytwo{\dchi} \Ytwo{\chi_2-\chi_1}
    \end{align*}
    with
    \[
        \omega_2(R) = (8 + \hat\B) R.
    \]
    This directly yields the claim, as
    \begin{align}
        \|(F'(\chi_2) - F'(\chi_1))[\chi_2-\chi_1]\|_{(Y^2)^*}
        &= \underset{\Ytwo{\dchi} = 1}\sup |(F'(\chi_2) - F'(\chi_1)) [\chi_2-\chi_1, \dchi]|  \notag\\
        &\le \omega_2 \Ytwo{\chi_2-\chi_1}.
    \end{align}
\end{proof}

\subsection{Convergence of Newton's Method}

We are now ready to connect the results outlined above to prove that the Newton-KKT method applied to the free flight optimization problem \eqref{eq:reduced-problem} converges to a global minimizer as characterized in \Cref{sec:optimality-conditions} provided that there is a $u\in\R^2$ with $\|u\|=1$ such that $u^T\glob\xt \ge c$. Roughly speaking, the optimal route needs to head towards the destination, dominating any route that involves flying the opposite direction. It is intuitively clear that this holds even for relatively strong wind conditions.

\begin{theorem} \label{th:newton-convergence}
    Let $\glob\chi = (\glob z, \glob \lambda)$ be a global solution of \eqref{eq:reduced-problem} that satisfies the first and second order conditions for optimality with $\uB>0$. Moreover let there be a $c>0$ and a $u\in\R^2$ with $\|u\|=1$ such that $u^T\glob\xt \ge c$ for almost all $\tau\in(0,1)$. Finally, let $\omega:=\omega_1\omega_2$, as given in \Cref{th:Finv-upper-bound,th:dF2-dF1-upper-bound}.
    \\
    Then there is a $R_C>0$, such that the ordinary Newton iterates defined in \Cref{sec:newton} converge to $\glob\chi$ at an estimated rate
    \begin{equation} \label{eq:newton-convergence-rate}
        \Ytwo{\chi^{k+1}-\glob\chi} \le \frac\omega2 \Ytwo{\chi^k - \glob\chi},
    \end{equation}
    if initialized with $\chi^0\in\Nhood[\glob\chi][R_C]{}$ and provided that the iterates $\chi^k$ remain in $\Nhood[\glob\chi][R_C]{}$. Moreover, $\glob\chi$ is unique in $\Nhood[\glob\chi][R_C]{}$.
\end{theorem}

\begin{proof}

    In \Cref{th:inf-sup-condition,th:Lzz-lower-bound-neighborhood,th:Lzz-upper-bound} we showed that the \emph{inf-sup} condition is satisfied, that, $\L_{zz}(\chi)$ is positive definite on the kernel of the constraint for all $x\in\Nhood[\glob\chi][R_C]{}$, and that it is bounded from above. Consequently, $F'(\chi)$ is invertible with
    \begin{align*}
        \|F'(\chi)^{-1}\|_{(Y^2)^* \to Y^2} \le \omega_1 \qquad \forall~\chi\in\Nhood[\glob\chi][R_C]{},
    \end{align*}
    as confirmed in \Cref{th:Finv-upper-bound}.
    Further, it follows from \Cref{th:Finv-upper-bound,th:dF2-dF1-upper-bound} that
    \renewcommand{\myindent}{\hspace{3cm}}
    \begin{align*}
        &\Ytwo{F'(\chi_1)^{-1} (F'(\chi_2) - F'(\chi)) [\chi_2-\chi_1]} \\
        &\myindent\le \|{F'(\chi_1)^{-1}}\|_{(Y^2)^* \to Y^2} \|(F'(\chi_2) - F'(\chi_1)) [\chi_2-\chi_1]\|_{(Y^2)^*} \\
        &\myindent\le \omega_1 \omega_2 \Ytwo{\chi_2-\chi_1} \\
        &\myindent\le \omega \Ytwo{\chi_2-\chi_1}
    \end{align*}
    for $\chi_1,\chi_2 \in \Nhood[\glob\chi][R_C]{}$.
    It is clear that since $\omega_1$ is bounded and $\omega_2 = (8 + \hat\B) R$, there is a $R_C>0$ such that $\omega := \omega_1\omega_2 < 2$.
    We now define $e_k := \chi^k - \glob\chi$ and proceed for $\mu \in (0,1)$ as follows:
    \renewcommand{\myindent}{\hspace{1cm}}
    \begin{align*}
        &\Ytwo{\chi^k + \mu \Dchi^k - \glob\chi} \\
        &\myindent= \Ytwo{ e_k - \mu F'(\chi^k)^{-1} F(\chi^k) } \\
        &\myindent= \Ytwo{ e_k - \mu F'(\chi^k)^{-1} (F(\chi^k)-\underbrace{F(\glob\chi)}_{=0}) } \\
        &\myindent= \Ytwo{ (1-\mu) e_k - \mu F'(\chi^k)^{-1} \int_{s=0}^1 \left(F'(\chi^k - s e_k) - F'(\chi^k)\right) e_k \, ds \; } \\
        &\myindent\le (1-\mu) \, \Ytwo{e_k} + \frac{\mu}{2} \omega \, \Ytwo{e_k},
    \end{align*}
    which yields the claim with $\mu=1$ as
    \begin{align*}
        \Ytwo{e_{k+1}} \le \frac\omega2 \Ytwo{e_k}.
    \end{align*}
    In order to prove uniqueness in $\Nhood[\glob\chi][R_C]{}$, assume there is a second solution $\loc\chi\neq\glob\chi$ with $F(\loc\chi)=0$ and $\loc\chi\in\Nhood[\glob\chi][R_C]{}$. Initialized with $\chi^0 := \loc\chi$ it certainly holds that $\chi^1 = \loc\chi$. However, from \eqref{eq:newton-convergence-rate} we obtain
    \begin{equation*}
        \Ytwo{\chi^1-\glob\chi} \le \frac\omega2 \Ytwo{\chi^0 - \glob\chi} < \Ytwo{\chi^0 - \glob\chi},
    \end{equation*}
    due to $\omega<2$, which yields a contradiction.
\end{proof}

\section{Conclusion}

It has been demonstrated that the Newton-KKT method can be used to solve the free flight trajectory optimization problem under certain conditions. These conditions are
i) the requirement for the iterates to remain within a $L^\infty$-neighborhood of the solution, and ii) a starting point that is sufficiently close to the solution. Such a suitable starting point can be found efficiently by calculating shortest paths on a specific graph~\cite{BorndoerferDaneckerWeiser2022a}. Hence an important tool for efficient deterministic global optimization of the free flight problem has been established.

\section*{Declarations}

\subsection*{Funding}
    This research was funded by the DFG Research Center of Excellence MATH$^+$ -- Berlin Mathematics Research Center, Project TrU-4.

\subsection*{Competing interests}
    The authors have no relevant financial or non-financial interests to disclose.

\subsection*{Ethics approval}
    Not applicable.

\subsection*{Consent to participate}
    Not applicable.

\subsection*{Data Availability}
    Data sharing not applicable to this article as no datasets were generated or analysed during the current study.

\subsection*{Consent for publication}
    We confirm that all authors agree with the submission of this manuscript to \emph{Public Transport Optimization: From Theory to Practice}.

\subsection*{Authors' contributions}
    Conceptualization, R.B and M.W.;
    methodology, F.D. and M.W.;
    validation, F.D.;
    formal analysis, F.D. and M.W.;
    investigation, F.D. and M.W.;
    resources, R.B., F.D. and M.W.;
    writing-original draft preparation, F.D. and M.W.;
    writing-review and editing, R.B.;
    supervision, R.B.;
    project administration, R.B. and M.W.;
    funding acquisition, R.B. and M.W.;
    All authors have read and agreed to the published version of the manuscript.


\bibliography{main}

\pagebreak
\appendix
\section{Supplementary Material}

\subsection{Global Bounds}

The derivative $f=t_\tau$ of parametrized time as defined in \eqref{eq:dt-dtau} consists of two terms, the tailwind term
\begin{align}
    f_1 &= -\frac{\xt^T w}{g},  \label{eq:f1} \\
    g &= \ov^2 - w^T w, \label{eq:g}
\end{align}
and the length term
\begin{align} \label{eq:f2}
    f_2 =  g^{-1}\left((\xt^T w)^2 + g (\xt^T\xt)\right)^{1/2}.
\end{align}
%
At each time $\tau$, we obtain
\begin{align}
    \uv^2:=\ov^2 - \oc_0^2 \le g \le \ov^2.
\end{align}
The directional derivatives of $g$ in direction $\dx$ and $\Dx \in \dX$ read
\begin{align}
    g' \dx
    &= -2w^T w_x \dx \\
    \Rightarrow\quad \|g'\| &\le 2\oc_0 \oc_1
    \\[2ex]
    \dx^T g'' \dx
    &= - 2\dx^T w_x^T w_x \dx
       - 2w_{xx}[w,\dx,\dx] \\
    \Rightarrow\quad \|g''\| &\le 2(\oc_1^2+\oc_0\oc_2).
    \\[2ex]
    g'''[\dx,\dx,\Dx]
    &= - 6 w_{xx}[ w_x\dx,\dx,\Dx] -2 w_{xxx}[w,\dx,\dx,\Dx]) \\
    \Rightarrow\quad \|g'''\| &\le 2(3\oc_1\oc_2 + \oc_0\oc_3).
\end{align}
%
%
%
%
For the tailwind term, we consider
\begin{align} \label{eq:df1}
    f_1'(\xi,\xt)[\dx,\dxt]
    = g^{-2} (\xt^Tw) (g'\dx) - g^{-1} \xt^Tw_x\dx - g^{-1} w^T\dxt,
\end{align}
which is bounded by
\begin{align} \label{eq:df1-upper-bound}
    |f_1'(\xi,\xt)[\dx,\dxt]|
    &\le
    \left(2\frac{\oc_0^2 \oc_1}{\uv^4} + \frac{\oc_1}{\uv^2} \right) \nxt \ndx
    + \frac{\oc_0}{\uv^2} \ndxt.
\end{align}
%
The second directional derivatives is
\begin{alignat}{2} \label{eq:ddf1-mixed}
    f_1''(\xi,\xt)[\dx,\dxt][\Dx,\Dxt]
    &= 		 - 2g^{-3} (g'\Dx) (\xt^T w) (g'\dx)    \quad&
    &		 +  g^{-2} (\Dxt^T w) (g'\dx)           \notag\\
    &\quad	 +  g^{-2} (\xt^T w_x \Dx) (g'\dx)      \quad&
    &	     +  g^{-2} (\xt^T w) (\dx^Tg''\Dx)      \notag\\
    &\quad   +  g^{-2} (g'\Dx) (\xt^T w_x \dx)      \quad&
    &        -  g^{-1} (\Dxt w_x \dx)               \notag\\
    &\quad   -  g^{-1} w_{xx}[\xt,\dx,\Dx]          \quad&
    &        +  g^{-2} (g'\Dx) (w^T \dxt)           \notag\\
    &\quad   -  g^{-1} (\dxt^T w_x \Dx)
\end{alignat}
and in particular
\begin{alignat}{2} \label{eq:ddf1}
    f_1''(\xi,\xt)[\dx,\dxt]^2
    &= 		- 2g^{-3} (g'\dx)^2 (\xt^T w)         \quad&
    &		+ 2g^{-2} (g'\dx) (\xt^T w_x\dx)      \notag\\
    &\quad	+  g^{-2} (\dx^Tg''\dx) (\xt^T w)     \quad&
    &		-  g^{-1} w_{xx}[\xt,\dx,\dx]         \notag\\
    &\quad	- 2g^{-1} (\dxt^T w_x\dx)             \quad&
    &		+ 2g^{-2} (\dxt^T w) (g'\dx),
\end{alignat}
which yields
\begin{align} \label{eq:ddf1-upper-bound-mixed}
    |f_1''(\xi,\xt)[\dx,\dxt][\Dx,\Dxt]|
    &\le
    \bigg[
        8 \frac{\oc_0^3 \oc_1^2}{\uv^{6}}
        + 6 \frac{\oc_0 \oc_1^2}{\uv^{4}}
        + 2 \frac{\oc_0^2 \oc_2}{\uv^{4}}
        +   \frac{\oc_2}{\uv^{2}}
    \bigg] \nxt \ndx \nDx \notag\\
    &\quad+ \bigg[
        2\frac{\oc_0^2 \oc_1}{\uv^4}
        + \frac{\oc_1}{\uv^2}
    \bigg] \ndxt \nDx    \notag\\
    &\quad+ \bigg[
    	2 \frac{\oc_0^2 \oc_1}{\uv^4}
        + \frac{\oc_1}{\uv^2}
    \bigg] \ndx \nDxt
\end{align}
and
\begin{align} \label{eq:ddf1-upper-bound}
    |f_1''(\xi,\xt)[\dx,\dxt]^2|
    &\le
    \bigg[
        8 \frac{\oc_0^3 \oc_1^2}{\uv^{6}}
        + 6 \frac{\oc_0 \oc_1^2}{\uv^{4}}
        + 2 \frac{\oc_0^2 \oc_2}{\uv^{4}}
        +   \frac{\oc_2}{\uv^{2}}
    \bigg] \nxt \ndx^2 \notag\\
    &\quad+ \bigg[
    	4 \frac{\oc_0^2 \oc_1}{\uv^4}
        + 2\frac{\oc_1}{\uv^2}
    \bigg] \ndx \ndxt ,
\end{align}
respectively.
%
Finally, the third directional derivative is
\begin{alignat}{2} \label{eq:dddf1}
    &\hspace{-1cm}f_1'''(\xi,\xt)[\dx,\dxt]^2[\Dx,\Dxt] \notag\\
    =
    %
    &\quad	  6g^{-4} (g'\Dx) (g'\dx)^2 (\xt^T w)         \quad&
    &		- 4g^{-3} (g'\dx) (\Dx^Tg''\dx) (\xt^T w)     \notag\\
    &     	- 2g^{-3} (g'\Dx) (\dx^Tg''\dx) (\xt^T w)     \quad&
    &		+ g^{-2} g'''[\dx,\dx,\Dx] (\xt^T w)          \notag\\
    &     	- 2g^{-3} (g'\dx)^2 (\xt^T w_x \Dx)           \quad&
    &		+ g^{-2} (\dx^Tg''\dx) (\xt^T w_x \Dx)        \notag\\
    &     	- 4g^{-3} (g'\Dx) (g'\dx) (\xt^T w_x \dx)     \quad&
    &		+ 2g^{-2} (\Dx^T g''\dx) (\xt^T w_x \dx)      \notag\\
    &     	+ 2g^{-2} (g'\dx) w_{xx}[\xt,\dx,\Dx]         \quad&
    &		+ g^{-2} (g'\Dx) w_{xx}[\xt,\dx,\dx]          \notag\\
    &     	- g^{-1} w_{xxx}[\xt,\dx,\dx,\Dx]             \notag\\
    %
    &     	- 2g^{-3} (g'\dx)^2 (\Dxt^T w)                \quad&
    &		+ g^{-2} (\dx^Tg''\dx) (\Dxt^T w)             \notag\\
    &     	+ 2g^{-2} (g'\dx) (\Dxt^T w_x \dx)            \quad&
    &		- g^{-1} w_{xx}[\Dxt,\dx,\dx]                 \notag\\
    %
    &     	- 4g^{-3} (g'\Dx) (g'\dx) (\dxt^T w)          \quad&
    &		+ 2g^{-2} (\Dx^T g''\dx) (\dxt^T w)           \notag\\
    &     	+ 2g^{-2} (g'\dx) (\dxt^T w_x \Dx)            \quad&
    &		+ 2g^{-2} (g'\Dx) (\dxt^T w_x \dx)            \notag\\
    &     	- 2g^{-1} w_{xx} [\dxt,\dx,\Dx],
\end{alignat}
which is bounded by
%
\renewcommand{\myindent}{\hspace{1.5cm}}
\begin{align} \label{eq:dddf1-upper-bound}
    &\hspace{-1cm} |f_1'''(\xi,\xt)[\dx,\dxt]^2[\Dx,\Dxt]| \notag\\
    &\le
    \frac{\nxt}{\uv} \bigg[
        \frac{\oc_1^3}{\uv^3} \left(
             48\frac{\oc_0^4}{\uv^4}
            +48\frac{\oc_0^2}{\uv^2}
            + 6
        \right)
        \notag \\ &\myindent
        + \frac{\oc_1 \oc_2}{\uv^2} \left(
             24\frac{\oc_0^3}{\uv^3}
            +18\frac{\oc_0}{\uv}
        \right)
        \notag \\ &\myindent
        + \frac{\oc_3}{\uv} \left(
            2\frac{\oc_0^2}{\uv^2}
            + 1
        \right)
    \bigg] \ndx^2 \nDx
    \notag \\
    &\quad + \left[
        \frac{\oc_1^2}{\uv^3} \left(8\frac{\oc_0^3}{\uv^3} + 6\frac{\oc_0}{\uv}\right)
        + \frac{\oc_2}{\uv^2} \left(2\frac{\oc_0^2}{\uv^2} + 1 \right)
    \right] \ndx^2 \nDxt
    \notag \\
    &\quad + \left[
        \frac{\oc_1^2}{\uv^3} \left(
             16 \frac{\oc_0^3}{\uv^3}
            +12 \frac{\oc_0}{\uv}
        \right)
        + \frac{\oc_2}{\uv^2} \left(
              4 \frac{\oc_0^2}{\uv^2}
            + 2
        \right)
    \right] \ndx \ndxt \nDx.
\end{align}


Before we turn to the length term $f_2$, we first consider the term
\begin{align} \label{eq:F}
    F := (\xt^T w)^2 + g (\xt^T\xt)
\end{align}
with
\[
    \uv^2 \nxt^2 \le F \le \nxt^2 \oov^2.
\]
We also note that
\[
    \frac{g}{F} \le \frac1{\nxt^2}.
\]
%
%
Then
\begin{align} \label{eq:dF}
    F'(\xi,\xt)[\dx,\dxt] =
        &2 (\xt^T w) ((\dxt^T w) + (\xt^T w_x \dx)) \notag\\
        &+ (g'\dx) (\xt^T\xt) + 2g(\xt^T \dxt),
\end{align}
which is bounded by
\begin{align} \label{eq:dF-upper-bound}
    |F'(\xi,\xt)[\dx,\dxt]| \le
    2 \oov^2 \nxt \ndxt
    + 4 \oc_0 \oc_1 \nxt^2 \ndx,
\end{align}
%
The second derivative is
\begin{alignat}{2} \label{eq:ddF-mixed}
    &\hspace{-1.5cm}F''(\xi,\xt)[\dx,\dxt][\Dx,\Dxt] \notag\\
    &=		  2 (\xt^T w) (\dxt^T w_x \Dx)      &\quad
    &		+ 2 (\xt^T w_x\Dx) (\dxt^T w)       \notag\\
    &\quad	+ 2 (\Dxt^T w) (\dxt^T w)           &\quad
    &		+ 2 (\xt^T w_x \Dx) (\xt^T w_x \dx) \notag\\
    &\quad	+ 2 (\xt^T w) w_{xx}[\xt,\dx,\Dx]   &\quad
    &		+ 2 (\Dxt^T w) (\xt^T w_x \dx)      \notag\\
    &\quad	+ 2 (\xt^T w) (\Dxt^T w_x \dx)      &\quad
    &	  	+ (\Dx^T g''\dx) (\xt^T\xt)         \notag\\
    &\quad	+ 2(g'\dx) (\Dxt^T\xt)              &\quad
    &		+ 2(g'\Dx) (\xt^T \dxt)             \notag\\
    &\quad	+ 2g(\Dxt^T \dxt)
\end{alignat}
and in particular
\begin{alignat}{2} \label{eq:ddF}
    F''(\xi,\xt)[\dx,\dxt]^2
    &=		  4 (\xt^T w) (\dxt^T w_x \dx)      &\quad
    &		+ 4 (\dxt^T w) (\xt^T w_x \dx)      \notag\\
    &\quad	+ 2 (\dxt^T w)^2                    &\quad
    & 		+ 2 (\xt^T w_x \dx)^2               \notag\\
    &\quad	+ 2 (\xt^T w) w_{xx}[\xt,\dx,\dx]   &\quad
    &	  	+ (\dx^T g''\dx) (\xt^T\xt)         \notag\\
    &\quad	+ 4(g'\dx) (\dxt^T\xt)              &\quad
    &		+ 2g(\dxt^T \dxt),
\end{alignat}
which yields
\begin{align} \label{eq:ddF-upper-bound-mixed}
    |F''(\xi,\xt)[\dx,\dxt][\Dx,\Dxt]| &\le
    \left(4 \oc_1^2 + 4\oc_0 \oc_2\right) \nxt^2 \ndx \nDx \notag\\
	&\quad + 8 \oc_0 \oc_1 \nxt \ndx \nDxt \notag\\
	&\quad + 8 \oc_0 \oc_1 \nxt \ndxt \nDx \notag\\
	&\quad + 2 \oov^2 \ndxt \nDxt
\end{align}
and
\begin{align} \label{eq:ddF-upper-bound}
    |F''(\xi,\xt)[\dx,\dxt]^2| &\le
    \left(4 \oc_1^2 + 4\oc_0 \oc_2\right) \nxt^2 \ndx^2 \notag\\
	&\quad +16 \oc_0 \oc_1 \nxt \ndx \ndxt \notag\\
	&\quad + 2 \oov^2 \ndxt^2 ,
\end{align}
respectively.
%
%
The third derivative is
\begin{align} \label{eq:dddF}
    &\quad F'''(\xi,\xt)[\dx,\dxt]^2 [\Dx,\Dxt] \notag\\
    &=		  4 (\Dxt^T w) (\dxt^T w_x \dx)         &
    &		+ 4 (\xt^T w_x \Dx) (\dxt^T w_x \dx)    \notag\\
    &\quad	+ 4 (\xt^T w) w_{xx}[\dxt,\dx,\Dx]      &
    &		+ 4 (\dxt^T w_x \Dx) (\xt^T w_x \dx)    \notag\\
    &\quad	+ 4 (\dxt^T w) (\Dxt^T w_x \dx)         &
    &		+ 4 (\dxt^T w) w_{xx}[\xt,\dx,\Dx]      \notag\\
    &\quad	+ 4 (\dxt^T w) (\dxt^T w_x \Dx)         &
    &		+ 4 (\xt^T w_x \dx) (\Dxt^T w_x \dx)    \notag\\
    &\quad	+ 4 (\xt^T w_x \dx) w_{xx}[\xt,\dx,\Dx] &
    &		+ 2 (\Dxt^T w) w_{xx}[\xt,\dx,\dx]      \notag\\
    &\quad	+ 2 (\xt^T w_x \Dx) w_{xx}[\xt,\dx,\dx] &
    &		+ 2 (\xt^T w) w_{xxx}[\xt,\dx,\dx,\Dx]  \notag\\
    &\quad	+ 2 (\xt^T w) w_{xx}[\Dxt,\dx,\dx]      &
    &	  	+   g'''[\dx,\dx,\Dx] (\xt^T\xt)        \notag\\
    &\quad	+ 2 (\dx^T g''\dx) (\Dxt^T \xt)         &
    &		+ 4 (\Dx^T g''\dx) (\dxt^T \xt)         \notag\\
    &\quad	+ 4 (g'\dx) (\dxt^T \Dxt)               &
    &		+ 2 (g'\Dx) (\dxt^T \dxt),
\end{align}
which is bounded by
\begin{align} \label{eq:ddfF-upper-bound}
    |F'''(\xi,\xt)[\dx,\dxt]^2 [\Dx,\Dxt]|
    &\le	  4 \nxt^2 (\oc_0 \oc_3 +3 \oc_1 \oc_2)\ndx^2       \nDx  \notag\\
    &\quad	+ 8 \nxt   (\oc_1^2+\oc_0\oc_2)        \ndx^2       \nDxt \notag\\
    &\quad	+16 \nxt   (\oc_1^2+\oc_0\oc_2)        \ndx \ndxt   \nDx  \notag\\
    &\quad	+16        \oc_0 \oc_1                 \ndx \ndxt   \nDxt \notag\\
    &\quad	+ 8        \oc_0 \oc_1                      \ndxt^2 \nDx.
\end{align}
%
%
%
%
For the length term $f_2 = g^{-1} \sqrt{F}$, we thus obtain
\begin{align} \label{eq:df2}
    f_2'(\xi,\xt)[\dx,\dxt] = -g^{-2} (g'\dx) F^{1/2}
        + \frac12 g^{-1} F^{-1/2} F'[\dx,\dxt],
\end{align}
which is bounded by
\begin{align} \label{eq:df2-upper-bound}
    |f_2'(\xi,\xt)[\dx,\dxt]|
    &\le
        \left(2 \frac{\oc_0 \oc_1 \oov}{\uv^4} + 4 \frac{\oc_0 \oc_1}{\uv^3}\right) \nxt \ndx
        + 2 \uv^{-3} \oov^2 \ndxt.
\end{align}
%
The second derivative is
\begin{alignat}{2} \label{eq:ddf2-mixed}
    f_2''(\xi,\xt)[\dx,\dxt][\Dx,\Dxt] =
    &\quad 	 2g^{-3} (g'\Dx) (g'\dx) F^{1/2}                       \notag\\
    &		- g^{-2} (\dx^T g''\Dx) F^{1/2}                        \notag\\
    &     	- \frac12 g^{-2} (g'\dx) F^{-1/2} F'[\Dx,\Dxt]        \notag\\
    &       - \frac12 g^{-2} (g'\Dx) F^{-1/2} F'[\dx,\dxt]         \notag\\
    &     	+ \frac12 g^{-1} F^{-1/2} F''[\dx,\dxt][\Dx,\Dxt]     \notag\\
    &    	- \frac14 g^{-1} F^{-3/2} F'[\dx,\dxt] F'[\Dx,\Dxt]
\end{alignat}
and in particular
\begin{alignat}{2} \label{eq:ddf2}
    f_2''(\xi,\xt)[\dx,\dxt]^2 =
    &\quad 	 2g^{-3} (g'\dx)^2 F^{1/2}                      \notag\\
    &		- g^{-2} (\dx^T g''\dx) F^{1/2}                 \notag\\
    &     	- g^{-2} (g'\dx) F^{-1/2} F'[\dx,\dxt]          \notag\\
    &		+ \frac12 g^{-1} F^{-1/2} F''[\dx,\dxt]^2       \notag\\
    &     	- \frac14 g^{-1} F^{-3/2} (F'[\dx,\dxt])^2,
\end{alignat}
which yields
\begin{align} \label{eq:ddf2-upper-bound-mixed}
    &\hspace{-1cm}|f_2''(\xi,\xt)[\dx,\dxt][\Dx,\Dxt]| \notag\\
    &\le \left[
        8   \frac{\oc_0^2 \oc_1^2 \oov}{\uv^6}
	    +12 \frac{\oc_0^2 \oc_1^2}{\uv^5}
	    + 2 \frac{(\oc_1^2 + \oc_0 \oc_2) \oov}{\uv^4}
        + 2 \frac{\oc_1^2 + \oc_0 \oc_2}{\uv^3}
	\right] \nxt \ndx \nDx \notag \\
	&\quad+ \left[
	    4 \frac{\oc_0 \oc_1 \oov^2}{\uv^5}
	    + 4 \frac{\oc_0 \oc_1}{\uv^3}
	\right] \ndx \nDxt \notag\\
	&\quad+ \left[
	    4 \frac{\oc_0 \oc_1 \oov^2}{\uv^5}
	    + 4 \frac{\oc_0 \oc_1}{\uv^3}
	\right] \ndxt \nDx \notag\\
	&\quad+ \left[
	    \frac{\oov^4}{\uv^5}
	    + \frac{\oov^2}{\uv^3}
	\right] \nxt^{-1} \ndxt \nDxt
\end{align}
and
\begin{align} \label{eq:ddf2-upper-bound}
    &\hspace{-1cm}|f_2''(\xi,\xt)[\dx,\dxt]^2| \notag\\
    &\le \left[
        8   \frac{\oc_0^2 \oc_1^2 \oov}{\uv^6}
        +12 \frac{\oc_0^2 \oc_1^2}{\uv^5}
        + 2 \frac{(\oc_1^2 + \oc_0 \oc_2) \oov}{\uv^4}
        + 2 \frac{\oc_1^2 + \oc_0 \oc_2}{\uv^3}
    \right] \nxt \ndx^2 \notag \\
    &\quad+ \left[
        8 \frac{\oc_0 \oc_1 \oov^2}{\uv^5}
        + 8 \frac{\oc_0 \oc_1}{\uv^3}
    \right] \ndx \ndxt \notag\\
    &\quad+ \left[
        \frac{\oov^4}{\uv^5}
        + \frac{\oov^2}{\uv^3}
    \right] \nxt^{-1} \ndxt^2
\end{align}
%
%
The third derivative is
\begin{alignat}{2} \label{eq:dddf2}
    f_2'''(\xi,\xt)[\dx,\dxt]^2 [\Dx,\Dxt]
    =& 		-6g^{-4} (g'\Dx) (g'\dx)^2 F^{1/2}                               \notag\\    
    &		+4g^{-3} (g'\dx) (\Dx^Tg''\dx) F^{1/2}                           \notag\\  
    &     	+ g^{-3} (g'\dx)^2 F^{-1/2} F'[\Dx,\Dxt]                        \notag\\    
    &		+2g^{-3} (g'\Dx) (\dx^T g''\dx) F^{1/2}                          \notag\\  
    &     	- g^{-2} g'''[\dx,\dx,\Dx] F^{1/2}                               \notag\\    
    &		-\frac12 g^{-2} (\dx^T g''\dx) F^{-1/2} F'[\Dx,\Dxt]            \notag\\
    &     	+        g^{-3} (g'\Dx) (g'\dx) F^{-1/2} F'[\dx,\dxt]            \notag\\
    &		-\frac12 g^{-2} (\Dx^Tg''\dx) F^{-1/2} F'[\dx,\dxt]              \notag\\
    &     	+\frac14 g^{-2} (g'\dx) F^{-3/2} F'[\dx,\dxt]  F'[\Dx,\Dxt]     \notag\\
    &		-\frac12 g^{-2} (g'\dx) F^{-1/2} F''[\dx,\dxt][\Dx,\Dxt]        \notag\\
    &     	+         g^{-3} (g'\Dx) (g'\dx) F^{-1/2} F'[\dx,\dxt]           \notag\\
    &		- \frac12 g^{-2} (\Dx g''\dx) F^{-1/2} F'[\dx,\dxt]              \notag\\
    &     	+ \frac14 g^{-2} (g'\dx) F^{-3/2} F'[\dx,\dxt] F'[\Dx,\Dxt]     \notag\\
    &		- \frac12 g^{-2} (g'\dx) F^{-1/2} F''[\dx,\dxt][\Dx,\Dxt]       \notag\\
    &     	+ \frac14 g^{-2} (g'\Dx) F^{-3/2} (F'[\dx,\dxt])^2               \notag\\
    &		+ \frac38 g^{-1} F^{-5/2} (F'[\dx,\dxt])^2 F'[\Dx,\Dxt]         \notag\\
    &     	- \frac12 g^{-1} F^{-3/2} F'[\dx,\dxt] F''[\dx,\dxt][\Dx,\Dxt]  \notag\\
    &		- \frac12 g^{-2} (g'\Dx) F^{-1/2} F''[\dx,\dxt]^2                \notag\\
    &     	- \frac14 g^{-1} F^{-3/2} F''[\dx,\dxt]^2 F'[\Dx,\Dxt]          \notag\\
    &		+ \frac12 g^{-1} F^{-1/2} F'''[\dx,\dxt]^2[\Dx,\Dxt],
\end{alignat}
which is bounded by
\renewcommand{\myindent}{\hspace{1.5cm}}
\begin{align} \label{eq:dddf2-upper-bound}
    &\hspace{-1cm}|f_2'''(\xi,\xt)[\dx,\dxt]^2 [\Dx,\Dxt]| \notag\\
    &\le \frac{2\nxt}{\uv} \bigg[
                \frac{\oc_3}{\uv} \left(
                     \frac{\oc_0}{\uv}
                    + \frac{\oov\oc_0}{\uv^2}
                \right)
                \notag \\ &\myindent
                + \frac{3\oc_1 \oc_2}{\uv^2} \left(
                      1
                    +   \frac{\oov}{\uv}
                    + 6 \frac{\oc_0^2}{\uv^2}
                    + 4 \frac{\oov \oc_0^2}{\uv^3}
                \right)
                \notag \\ &\myindent
                + \frac{6\oc_1^3}{\uv^3} \left(
                     3 \frac{\oc_0}{\uv}
                    +2\frac{\oov \oc_0}{\uv^2}
                    +8\frac{\oc_0^3}{\uv^3}
                    +4\frac{\oov \oc_0^3}{\uv^4}
                \right)
    \bigg] \ndx^2 \nDx \notag\\
    &\quad + \frac4{\uv} \bigg[
                \frac{\oc_1^2}{\uv^2}  \left(
                      1
                    +  \frac{\ov^2}{\uv^2}
                    + 9\frac{\oc_0^2}{\uv^2}
                    + 7\frac{\oc_0^2 \oov^2}{\uv^4}
                \right)
                \notag \\ &\myindent
                +
                \frac{\oc_2}{\uv} \left(
                      \frac{\oc_0}{\uv}
                    + \frac{\ov^2 \oc_0}{\uv^3}
                    + \frac{\oc_0^3}{\uv^3}
                \right)
    \bigg] \ndx^2 \nDxt  \notag\\
    &\quad + \frac8{\uv} \bigg[
                \frac{\oc_1^2}{\uv^2} \left(
                    1
                    +  \frac{\ov^2}{\uv^2}
                    + 9\frac{\oc_0^2}{\uv^2}
                    + 7\frac{\oc_0^2 \oov^2}{\uv^4}
                \right)
                \notag \\ &\myindent
                + \frac{\oc_2}{\uv} \left(
                      \frac{\oc_0}{\uv}
                    + \frac{\ov^2\oc_0}{\uv^3}
                    + \frac{\oc_0^3}{\uv^3}
                \right)
    \bigg] \ndx \ndxt \nDx \notag\\
    &\quad + \frac{8\oc_0\oc_1}{\nxt\uv^3} \left(
                1
                +3\frac{\oov^2}{\uv^2}
                +2\frac{\oov^4}{\uv^4}
    \right) \ndx \ndxt \nDxt \notag\\
    &\quad + \frac{4\oc_0\oc_1}{\nxt\uv^3} \left(
                1
                +3\frac{\oov^2}{\uv^2}
                +2\frac{\oov^4}{\uv^4}
    \right) \ndxt^2 \nDx \notag\\
    &\quad + \frac{3\oov^4}{\nxt^2\uv^5} \left(
                1
                + \frac{\oov^2}{\uv^2}
    \right) \ndxt^2 \nDxt.
\end{align}

\newcounter{theoremback}
\setcounter{theoremback}{\value{theorem}}
\setcounter{theorem}{\value{lemma-dddf-upper-bound-stored}}
\renewcommand{\myindent}{\hspace{2.5cm}}
\begin{lemma}
    Let $\|w(p)\| \le \oc_0 \le \ov/\sqrt5$, $\|w_x(p)\| \le \oc_1$, $\|w_{xx}(p)\| \le \oc_2$, and $\|w_{xxx}(p)\| \le \oc_3$ for every $p\in\Omega$. Moreover let $\uv^2 := \ov^2 - \oc_0^2$ and $\oov^2 := \ov^2 + \oc_0^2$. Then, for any $\xi\in X$, the third directional derivative of $f$ as given in~\eqref{eq:dt-dtau} is bounded by
    \renewcommand{\myindent}{\hspace{2cm}}
    \begin{alignat*}{2}
        \hspace{3cm}&\hspace{-3cm}|f'''(\xi,\xt) [\dx,\dxt]^2[\Dx,\Dxt]| \\
        \le~&\bigg(\nxt \og_0 \ndx^2 + \og_2 \ndx \ndxt + \frac{\og_4}{\nxt} \ndxt^2&&\bigg)~ \nDx \notag\\
        +&\bigg(\og_1 \ndx^2 + \frac{\og_3}{\nxt} \ndx \ndxt +\frac{\og_5}{\nxt^2} \ndxt^2 &&\bigg)~ \nDxt
    \end{alignat*}
    with
    \begin{align*}
        \og_0 &= \frac{2}{\uv^4} \left(
            37 \oc_1^3
            +21 \oc_1 \oc_2 \uv
            +2 \oc_3 \uv^2
        \right),
        & \og_3 &= 40\frac{\oc_1}{\uv^2 \nxt},
        \\
        \og_1 &= \frac1{\uv^3} \left(29 \oc_1^2 + 7 \uv \oc_2\right),
        & \og_4 &= 20\frac{\oc_1}{\uv^2 \nxt},
        \\
        \og_2 &=\frac1{\uv^3} (57 \oc_1^2 + 13\uv \oc_2),
        & \og_5 &= 18 \frac1{\uv \nxt^2}.
    \end{align*}
\end{lemma}

\begin{proof}
    We obtain $f$ by adding $f_1$ and $f_2$. The third derivative of $f$ can thus be bounded using~\eqref{eq:dddf1-upper-bound}, \eqref{eq:dddf2-upper-bound}, and the triangle inequality.

    \renewcommand{\myindent}{\hspace{1.5cm}}
    \begin{align*}
        &\quad|f'''[\dx,\dxt]^2 [\Dx,\Dxt]| \\
        &\le
        \frac{\nxt}{\uv} \bigg[
            \frac{\oc_3}{\uv} \left(
                1
                + 2\frac{\oc_0}{\uv}
                + 2\frac{\oov\oc_0}{\uv^2}
                + 2\frac{\oc_0^2}{\uv^2}
            \right)
            \notag \\ &\myindent
            + 6\frac{\oc_1 \oc_2}{\uv^2} \left(
                1
                + 1\frac{\oov}{\uv}
                + 3\frac{\oc_0}{\uv}
                + 6\frac{\oc_0^2}{\uv^2}
                + 4\frac{\oov \oc_0^2}{\uv^3}
                + 4\frac{\oc_0^3}{\uv^3}
            \right)
            \notag \\ &\myindent
            + 6\frac{\oc_1^3}{\uv^3} \left(
                  1
                + 6\frac{\oc_0}{\uv}
                + 4\frac{\oov \oc_0}{\uv^2}
                + 8\frac{\oc_0^2}{\uv^2}
                +16\frac{\oc_0^3}{\uv^3}
                + 8\frac{\oov \oc_0^3}{\uv^4}
                + 8\frac{\oc_0^4}{\uv^4}
            \right)
        \bigg] \ndx^2 \nDx \notag\\
        &\quad + \frac1{\uv} \bigg[
            2\frac{\oc_1^2}{\uv^2} \left(
                2
                + 3\frac{\oc_0}{\uv}
                + 2\frac{\ov^2}{\uv^2}
                +18\frac{\oc_0^2}{\uv^2}
                + 4\frac{\oc_0^3}{\uv^3}
                +14\frac{\oc_0^2 \oov^2}{\uv^4}
            \right)
            \notag \\ &\myindent
            + \frac{\oc_2}{\uv} \left(
                1
                + 4\frac{\oc_0}{\uv}
                + 2\frac{\oc_0^2}{\uv^2}
                + 4\frac{\ov^2 \oc_0}{\uv^3}
                + 4\frac{\oc_0^3}{\uv^3}
            \right)
        \bigg] \ndx^2 \nDxt  \notag\\
        &\quad + \frac1{\uv} \bigg[
            4 \frac{\oc_1^2}{\uv^2} \left(
                2
                + 4\frac{\oc_0^3}{\uv^3}
                + 3\frac{\oc_0}{\uv}
                + 2\frac{\ov^2}{\uv^2}
                +18\frac{\oc_0^2}{\uv^2}
                +14\frac{\oc_0^2 \oov^2}{\uv^4}
                \right)
            \notag \\ &\myindent
            + 2\frac{\oc_2}{\uv} \left(
                1
                + 4\frac{\oc_0}{\uv}
                + 2 \frac{\oc_0^2}{\uv^2}
                + 4\frac{\ov^2\oc_0}{\uv^3}
                + 4\frac{\oc_0^3}{\uv^3}
            \right)
        \bigg] \ndx \ndxt \nDx \notag\\
        &\quad + \frac{8\oc_0\oc_1}{\nxt\uv^3} \left(
                    1
                    +3\frac{\oov^2}{\uv^2}
                    +2\frac{\oov^4}{\uv^4}
        \right) \ndx \ndxt \nDxt \notag\\
        &\quad + \frac{4\oc_0\oc_1}{\nxt\uv^3} \left(
                    1
                    +3\frac{\oov^2}{\uv^2}
                    +2\frac{\oov^4}{\uv^4}
        \right) \ndxt^2 \nDx \notag\\
        &\quad + \frac{3\oov^4}{\nxt^2\uv^5} \left(
                    1
                    + \frac{\oov^2}{\uv^2}
        \right) \ndxt^2 \nDxt.
    \end{align*}
    With $\frac{\oc_0}{\ov} \le \frac1{\sqrt 5}$, we note that
    \[
        \frac{\oc_0}{\uv} \le \frac12, \qquad
        \frac{\oov}{\uv}  \le \sqrt{\frac32}, \quad\text{and} \qquad
        \frac{\ov}{\uv}   \le \frac{\sqrt5}{2}
    \]
    and obtain
    \renewcommand{\myindent}{\hspace{1.5cm}}
    \begin{align*}
        |f'''[\dx,\dxt]^2 [\Dx,\Dxt]|
        &\le
        \frac{\nxt}{\uv} \bigg[
            \frac{\oc_3}{\uv} \left(
                \frac52
                + \sqrt{\frac32}
            \right)
            \notag \\ &\myindent
            + 6\frac{\oc_1 \oc_2}{\uv^2} \left(
                \frac92
                + 2\sqrt{\frac32}
            \right)
            \notag \\ &\myindent
            + 6\frac{\oc_1^3}{\uv^3} \left(
                  \frac{17}{2}
                + 3 \sqrt{\frac32}
            \right)
        \bigg] \ndx^2 \nDx \notag\\
        &\quad + \frac1{\uv} \bigg[
            \frac{57}{2} \frac{\oc_1^2}{\uv^2}
            + \frac{13}{2} \frac{\oc_2}{\uv}
        \bigg] \ndx^2 \nDxt  \notag\\
        &\quad + \frac1{\uv} \bigg[
            57 \frac{\oc_1^2}{\uv^2}
            + 13\frac{\oc_2}{\uv}
        \bigg] \ndx \ndxt \nDx \notag\\
        &\quad + 40 \frac{\oc_1}{\nxt\uv^2} \ndx \ndxt \nDxt \notag\\
        &\quad + 20 \frac{\oc_1}{\nxt\uv^2} \ndxt^2 \nDx \notag\\
        &\quad + \frac{135}{8\nxt^2\uv} \ndxt^2 \nDxt,
    \end{align*}
    Rounding up the values yields the bound
    \begin{align*}
        |f'''[\dx,\dxt]^2 [\Dx,\Dxt]|
        &\le
        \frac{\nxt}{\uv} \bigg[
            4\frac{\oc_3}{\uv}
            + 42\frac{\oc_1 \oc_2}{\uv^2}
            + 74\frac{\oc_1^3}{\uv^3}
        \bigg] \ndx^2 \nDx \notag\\
        &\quad + \frac1{\uv} \bigg[
            29 \frac{\oc_1^2}{\uv^2}
            + 7 \frac{\oc_2}{\uv}
        \bigg] \ndx^2 \nDxt  \notag\\
        &\quad + \frac1{\uv} \bigg[
            57 \frac{\oc_1^2}{\uv^2}
            + 13\frac{\oc_2}{\uv}
        \bigg] \ndx \ndxt \nDx \notag\\
        &\quad + 40 \frac{\oc_1}{\nxt\uv^2} \ndx \ndxt \nDxt \notag\\
        &\quad + 20 \frac{\oc_1}{\nxt\uv^2} \ndxt^2 \nDx \notag\\
        &\quad + 18 \frac{1}{\nxt^2\uv} \ndxt^2 \nDxt.
    \end{align*}
\end{proof}
\setcounter{theorem}{\value{theoremback}}

\setcounter{theoremback}{\value{theorem}}
\setcounter{theorem}{\value{lemma-ddf-upper-bound-stored}}
\begin{lemma}
    Let $\|w(p)\| \le \oc_0 \le \ov/\sqrt5$, $\|w_x(p)\| \le \oc_1$, and $\|w_{xx}(p)\| \le \oc_2$ for every $p\in\Omega$. Moreover let $\uv^2 := \ov^2 - \oc_0^2$ and $\oov^2 := \ov^2 + \oc_0^2$. Then, for any $\xi\in X$, the second directional derivative of $f$ as given in~\eqref{eq:dt-dtau} is bounded by
    \begin{align*}
        |f''(\xi,\xt) [\dx,\dxt][\Dx,\Dxt]|
        \le
        &\ob_0 \nxt \ndx \nDx
        + \ob_1 \ndx \nDxt  \\
        &+ \ob_1 \ndxt \nDx
        + \ob_2 \nxt^{-1} \ndxt \nDxt
    \end{align*}
    with
    \[
        \ob_0 = 14 \frac{\oc_1^2}{\uv^3} + 4 \frac{\oc_2}{\uv^2}, \qquad
        \ob_1 = 7\frac{\oc_1}{\uv^2}, \qquad
    	\ob_2 = \frac4{\uv} .
    \]
\end{lemma}

\begin{proof}
    We obtain $f$ by adding $f_1$ and $f_2$. The second derivative of $f$ can thus be bounded using~\eqref{eq:ddf1-upper-bound-mixed}, \eqref{eq:ddf2-upper-bound-mixed}, and the triangle inequality.
    \begin{align*}
        |f''(\xi,\xt)[\dx,\dxt][\Dx,\Dxt]|
        &\le
        \bigg[
            8 \oc_0^2 \oc_1^2\frac{\oc_0 + \oov}{\uv^6}
            +12 \frac{\oc_0^2 \oc_1^2}{\uv^5} \\&\qquad
            + 2 \frac{\oc_1^2 \oov + \oc_0 \oc_2 \oov + 3 \oc_0 \oc_1^2 + \oc_0^2 \oc_2}{\uv^4} \\&\qquad
            + 2 \frac{\oc_1^2 + \oc_0 \oc_2}{\uv^3}
            +   \frac{\oc_2}{\uv^{2}}
        \bigg] \nxt \ndx \nDx \\
        &\quad+ \bigg[
    	    4 \frac{\oc_0 \oc_1 \oov^2}{\uv^5}
            + 2 \frac{\oc_0^2 \oc_1}{\uv^4}
    	    + 4 \frac{\oc_0 \oc_1}{\uv^3}
            + \frac{\oc_1}{\uv^2}
        \bigg] \ndxt \nDx    \\
        &\quad+ \bigg[
    	    4 \frac{\oc_0 \oc_1 \oov^2}{\uv^5}
        	+ 2 \frac{\oc_0^2 \oc_1}{\uv^4}
    	    + 4 \frac{\oc_0 \oc_1}{\uv^3}
            + \frac{\oc_1}{\uv^2}
        \bigg] \ndx       \nDxt   \\
    	&\quad+ \left[
    	    \frac{\oov^4}{\uv^5}
    	    + \frac{\oov^2}{\uv^3}
    	\right] \nxt^{-1} \ndxt \nDxt .
    \end{align*}
	With $\frac{\oc_0}{\ov} \le \frac1{\sqrt 5}$, we note that
	\[
		\frac{\oc_0}{\uv} \le \frac12, \qquad
		\frac{\oov}{\uv}  \le \sqrt{\frac32}, \quad\text{and} \qquad
		\frac{\ov}{\uv}   \le \frac{\sqrt5}{2}
	\]
	and obtain
    \begin{align*}
        |f''(\xi,\xt)[\dx,\dxt][\Dx,\Dxt]|
        &\le
        \left[
            \left(9 + 4 \sqrt{\frac32} \right)\frac{\oc_1^2}{\uv^3}
            + 4 \frac{\oc_2}{\uv^2}
        \right] \nxt \ndx \nDx \\
        &\quad+ \frac{13\oc_1}{2\uv^2} \ndxt \nDx    \\
        &\quad+ \frac{13\oc_1}{2\uv^2} \ndx \nDxt   \\
    	&\quad+ \frac{15}{4\uv} \nxt^{-1} \ndxt \nDxt .
    \end{align*}
    Rounding up the values yields the bound
    \begin{align*}
        |f''(\xi,\xt)[\dx,\dxt][\Dx,\Dxt]|
        &\le
        \bigg[
            14 \frac{\oc_1^2}{\uv^3}
            + 4 \frac{\oc_2}{\uv^2}
        \bigg] \nxt \ndx \nDx \\
        &\quad+ 7\frac{\oc_1}{\uv^2} \ndxt \nDx    \\
        &\quad+ 7\frac{\oc_1}{\uv^2} \ndx \nDxt   \\
    	&\quad+ \frac4{\uv} \nxt^{-1} \ndxt \nDxt .
    \end{align*}
\end{proof}
\setcounter{theorem}{\value{theoremback}}

\subsection{Bounds in a Neighborhood of a Minimizer}
Below we derive bounds that hold in a $L^\infty$-neighborhood of a global minimizer.
Let $\glob\chi = (\glob z, \glob\lambda)$ be a global minimizer of \eqref{eq:reduced-problem} and the corresponding Lagrange multipliers. Moreover, let $\chi_1,\chi_2 \in \Nhood[\glob\chi]$ and define $\Dchi := \chi_2-\chi_1$. Then it holds that $\Yinf{\Dchi} \le 2R$ and consequently
\begin{align}
    \Linf\Dx  &\underset{\eqref{eq:yinf-norm}}{\le} 2R, \\
    \Linf\Dxt &\underset{\eqref{eq:yinf-norm}}{\le} 2R.
\end{align}
Let $\Linf[\Omega]w {\le} \oc_0$,
$\Linf[\Omega]{w_x} {\le} \oc_1$,
$\Linf[\Omega]{w_{xx}} {\le} \oc_2$, and
$\Linf[\Omega]{w_{xxx}} {\le} \oc_3$,
then the following bounds hold,
\begin{alignat}{3}
    \|w(\xi_2) - w(\xi_1)\|
    &= |\int_0^1 w_x(\xi_1 + \mu \Dx)[\Dx] d\mu \; |  \;
    &\le \oc_1 \nDx
    &\le 2R \oc_1,  \\
    \|w_x(\xi_2) - w_{x,}(\xi_1)\|
    &= |\int_0^1 w_{xx}(\xi_1 + \mu \Dx)[\Dx]  d\mu \; |
    &\le \oc_2 \nDx
    &\le 2R \oc_2, \\
    \|w_{xx}(\xi_2) - w_{xx}(\xi_1)\|
    &= |\int_0^1 w_{xxx}(\xi_1 + \mu \Dx)[\Dx]  d\mu \; |
    &\le \oc_3 \nDx
    &\le 2R \oc_3.
\end{alignat}

\noindent
Moreover, we show that
\begin{align}
    |g(\xi_2)-g(\xi_1)|
    &= |\ov^2 - w(\xi_2)^Tw(\xi_2) - \ov^2 + w(\xi_1)^Tw(\xi_1)| \notag\\
    &= |w(\xi_2)^Tw(\xi_2) - w(\xi_1)^Tw(\xi_1)| \notag\\
    &\le 2 \oc_0 \oc_1 \nDx \notag\\
    &\le 4R \oc_0 \oc_1 \\[1.5ex]
    |g(\xi_2)^2-g(\xi_1)^2|
    &= |(g(\xi_2)-g(\xi_1))(g(\xi_2)+g(\xi_1))| \notag\\
    &\le (2 \oc_0 \oc_1 \nDx) (2\ov^2) \notag\\
    &\le 4 \oc_0 \oc_1 \ov^2 \nDx \notag\\
    &\le 8R \oc_0 \oc_1 \ov^2 \\[1.5ex]
    |g(\xi_2)^3 - g(\xi_1)^3|
    &= |g(\xi_2)-g(\xi_1)| \; |g(\xi_1)^2 + 2g(\xi_1)g(\xi_2) + g(\xi_2)^2| \notag\\
    &\le (2 \oc_0 \oc_1 \nDx) (4\ov^4) \notag\\
    &\le 8 \oc_0 \oc_1 \ov^4 \nDx \notag\\
    &\le 16R \oc_0 \oc_1 \ov^4 \\[1.5ex]
    |g'(\xi_2) - g'(\xi_1)|
    &= |\int_0^1 g''(\xi_1 + \mu \Dx)[\Dx] d\mu| \notag\\
    &\le 2(\oc_1^2 + \oc_0\oc_2) \nDx \notag\\
    &\le 4R (\oc_1^2 + \oc_0\oc_2) \\[1.5ex]
    \|g''(\xi_2) - g''(\xi_1)\|
    &= |\int_0^1 g'''(\xi_1 + \mu \Dx)[\Dx] d\mu| \notag\\
    &\le 2(3\oc_1\oc_2 + \oc_0\oc_3) \nDx \notag\\
    &\le 2R (3\oc_1\oc_2 + \oc_0\oc_3)
\end{align}
Furthermore, with $F$ as given in \eqref{eq:F}, and \eqref{eq:general-bounds} we get
\[
    \uv^2 (\glob L - R)^2 \le \uv^2 \nxt^2 \le F \le \nxt^2 \oov^2 \le (\glob L + R)^2 \oov^2
\]
and
\begin{align}
    |F'(\xi,\xt)[\Dx,\Dxt]|
    &\le 2 \oov^2 \nxt \nDxt
        + 4 \oc_0 \oc_1 \nxt^2 \nDx \notag\\
    &\le 2 \oov^2 (\glob L + R) \nDxt
        + 4 \oc_0 \oc_1 (\glob L + R)^2 \nDx.
\end{align}
This yields
\renewcommand{\myindent}{\hspace{3cm}}
\begin{align}
    &|F(\xi_2,\xi_{\tau,2})^{1/2} - F(\xi_1,\xi_{\tau,1})^{1/2}| \notag\\
    &\myindent\le \frac12 | \int_0^1 F(\xi_1+\mu\Dx)^{-1/2} F'(\xi_1+\mu\Dx) [\Dx,\Dxt] d\mu | \notag\\
    &\myindent\le \frac{\oov^2 (\glob L + R)}{\uv (\glob L - R)} \nDxt
        + \frac{2 \oc_0 \oc_1 (\glob L + R)^2}{\uv (\glob L - R)} \nDx \\[1.5ex]
    &|F(\xi_2,\xi_{\tau,2})^{-1/2} - F(\xi_1,\xi_{\tau,1})^{-1/2}| \notag\\
    &\myindent\le \frac12 | \int_0^1 F(\xi_1+\mu\Dx)^{-3/2} F'(\xi_1+\mu\Dx) d\mu \notag\\
    &\myindent\le \frac{\oov^2 (\glob L + R)}{\uv^3 (\glob L - R)^3} \nDxt
        + \frac{2 \oc_0 \oc_1 (\glob L + R)^2}{\uv^3 (\glob L - R)^3} \nDx\\[1.5ex]
    &|F(\xi_2,\xi_{\tau,2})^{-3/2} - F(\xi_1,\xi_{\tau,1})^{-3/2}|  \notag\\
    &\myindent\le \frac32 | \int_0^1 F(\xi_1+\mu\Dx)^{-5/2} F'(\xi_1+\mu\Dx) d\mu | \notag\\
    &\myindent\le \frac{\oov^2 (\glob L + R)}{\uv^5 (\glob L - R)^5} \nDxt
        + \frac{2 \oc_0 \oc_1 (\glob L + R)^2}{\uv^5 (\glob L - R)^5} \nDx
\end{align}

\noindent
For $f_1$ as defined in \eqref{eq:f1}, we obtain
\begin{alignat*}{2}
    &\hspace{-1cm}\left(f_1''(\xi_2,\xi_{\tau,2}) - f_1''(\xi_1,\xi_{\tau,1})\right))[\Dx,\Dxt][\dx,\dxt] \\
    = 
    & - 2g(\xi_2)^{-3} (g'(\xi_2)\dx) (\xi_{\tau,2}^T w(\xi_2)) (g'(\xi_2)\Dx)    \\
    & + 2g(\xi_1)^{-3} (g'(\xi_1)\dx) (\xi_{\tau,1}^T w(\xi_1)) (g'(\xi_1)\Dx)    \\
    & +  g(\xi_2)^{-3} g(\xi_2) (\dxt^T w(\xi_2)) (g'(\xi_2)\Dx)                    \\
    & -  g(\xi_1)^{-3} (\dxt^T w(\xi_1)) (g'(\xi_1)\Dx)                    \\
    & +  g(\xi_2)^{-3} (\xi_{\tau,2}^T w_x(\xi_2) \dx) (g'(\xi_2)\Dx)    \\
    & -  g(\xi_1)^{-3} (\xi_{\tau,1}^T w_x(\xi_1) \dx) (g'(\xi_1)\Dx)    \\
    & +  g(\xi_2)^{-3} (\xi_{\tau,2}^T w(\xi_2)) (\Dx^Tg''(\xi_2)\dx)      \\
    & -  g(\xi_1)^{-3} (\xi_{\tau,1}^T w(\xi_1)) (\Dx^Tg''(\xi_1)\dx)      \\
    & +  g(\xi_2)^{-3} (g'(\xi_2)\dx) (\xi_{\tau,2}^T w_x(\xi_2) \Dx)    \\
    & -  g(\xi_1)^{-3} (g'(\xi_1)\dx) (\xi_{\tau,1}^T w_x(\xi_1) \Dx)    \\
    & -  g(\xi_2)^{-3} (\dxt w_x(\xi_2) \Dx)                      \\
    & +  g(\xi_1)^{-3} (\dxt w_x(\xi_1) \Dx)                      \\
    & -  g(\xi_2)^{-3} w_{xx}(\xi_2)[\xi_{\tau,2},\Dx,\dx]          \\
    & +  g(\xi_1)^{-3} w_{xx}(\xi_1)[\xi_{\tau,1},\Dx,\dx]          \\
    & +  g(\xi_2)^{-3} (g'(\xi_2)\dx) (w(\xi_2)^T \Dxt)                    \\
    & -  g(\xi_1)^{-3} (g'(\xi_1)\dx) (w(\xi_1)^T \Dxt)                    \\
    & -  g(\xi_2)^{-3} (\Dxt^T w_x(\xi_2) \dx)                    \\
    & +  g(\xi_1)^{-3} (\Dxt^T w_x(\xi_1) \dx)
\end{alignat*}

\noindent
Using the bounds from above we finally obtain
\renewcommand{\myindent}{\hspace{1.5cm}}
\begin{multline} \label{eq:Delta-f1-bound}
    |\left(f_1''(\xi_2,\xi_{\tau,2}) - f_1''(\xi_1,\xi_{\tau,1})\right)[\Dx,\Dxt][\dx,\dxt]| \\
    \le \hat\beta_1 R \sqrt{\nDx^2+\nDxt^2} \sqrt{\ndx^2+\ndxt^2}
\end{multline}
with
\begin{align}
    \hat\beta_1 = \frac{4}{\uv^{12}} \big(
        &5
        +80 \oc_0 \oc_1 \ov^4
        + 8 \oc_0 \oc_1 \ov^2
        +12 \oc_0 \oc_1
        +16 \oc_0 \oc_2
        + 4 \oc_0 \oc_3  \notag\\
        &+16 \oc_1^2
        +12 \oc_1\oc_2
        + 4 \oc_1
        + 4 \oc_2
        + 2 \oc_3
    \big).
\end{align}

\noindent
For $f_2$ as defined in \eqref{eq:f2} we obtain
%
%
\begin{alignat*}{2}
    &\hspace{-1cm}(f_2''(\xi_2,\xi_{\tau,2})
    -f_2''(\xi_1,\xi_{\tau,1}))[\Dx,\Dxt][\dx,\dxt] \\ = 
        &~2 g(\xi_2)^{-3}(g'(\xi_2)\dx) (g'(\xi_2)\Dx) F(\xi_2)^{1/2}                      \\
        &-2 g(\xi_1)^{-3}(g'(\xi_1)\dx) (g'(\xi_1)\Dx) F(\xi_1)^{1/2}                      \\
        &- g(\xi_2)^{-2} (\Dx^T g''(\xi_2)\dx) F(\xi_2)^{1/2}                         \\
        &+ g(\xi_1)^{-2} (\Dx^T g''(\xi_1)\dx) F(\xi_1)^{1/2}                         \\
        &- \frac12 g(\xi_2)^{-2} (g(\xi_2)'\Dx) F(\xi_2)^{-1/2} F'(\xi_2)[\dx,\dxt]       \\
        &+ \frac12 g(\xi_1)^{-2} (g(\xi_1)'\Dx) F(\xi_1)^{-1/2} F'(\xi_1)[\dx,\dxt]       \\
        &- \frac12 g(\xi_2)^{-2} (g'(\xi_2)\dx) F(\xi_2)^{-1/2} F'(\xi_2)[\Dx,\Dxt]        \\
        &+ \frac12 g(\xi_1)^{-2} (g'(\xi_1)\dx) F(\xi_1)^{-1/2} F'(\xi_1)[\Dx,\Dxt]        \\
        &+ \frac12 g(\xi_2)^{-1} F(\xi_2)^{-1/2} F''(\xi_2)[\Dx,\Dxt][\dx,\dxt]      \\
        &- \frac12 g(\xi_1)^{-1} F(\xi_1)^{-1/2} F''(\xi_1)[\Dx,\Dxt][\dx,\dxt]      \\
        &- \frac14 g(\xi_2)^{-1} F(\xi_2)^{-3/2} F'(\xi_2)[\Dx,\Dxt] F'(\xi_2)[\dx,\dxt]  \\
        &+ \frac14 g(\xi_1)^{-1} F(\xi_1)^{-3/2} F'(\xi_1)[\Dx,\Dxt] F'(\xi_1)[\dx,\dxt]
\end{alignat*}

\noindent
Using the bounds from above, this yields
\begin{multline} \label{eq:Delta-f2-bound}
    |(f_2''(\xi_2,\xi_{\tau,2})
    -f_2''(\xi_1,\xi_{\tau,1}))[\Dx,\Dxt][\dx,\dxt]| \\
    \le \hat\beta_2 R \sqrt{\nDx^2 + \nDxt^2} \sqrt{\ndx^2 + \ndxt^2}
\end{multline}
with
\begin{align}
    \hat\beta_2 \le \frac{4}{\uv^{12}} \Bigg[
        &20
        +10 \oc_1
        + 7 \oc_2
        +   \oc_3
        +10 \oc_0 \oc_1
        +36 \oc_0 \oc_1 \ov^2
        +88 \oc_0 \oc_1 \ov^4  \notag\\
        &
        +20 \oc_0 \oc_2
        + 8 \oc_0 \oc_3
        +20 \oc_1^2
        +24 \oc_1\oc_2 \notag\\
        &
        + \left(
              \frac3{\uv (\glob L - R)}
            + \frac6{\uv^3 (\glob L - R)^3}
            + \frac6{\uv^5 (\glob L - R)^5}
        \right) \notag\\
        &\qquad\left(\oov^2 (\glob L + R) + 2 \oc_0 \oc_1 (\glob L + R)^2 \right)
    \Bigg].
\end{align}

\begin{lemma} \label{th:delta-f-bound}
    Let $\glob\xi,\glob L$ be a global minimizer of \eqref{eq:reduced-problem}. Moreover, let $\xi_1,\xi_2$ be given such that $\Cnorm{\xi_i - \glob\xi} \le R$ and define $\Delta\xi := \xi_1 - \xi_2$.
    Then there is a $\hat\B < \infty$ such that
    \begin{multline} \label{eq:delta-f-bound}
        |(f''(\xi_2,\xi_{\tau,2})
        -f''(\xi_1,\xi_{\tau,1}))[\Dx,\Dxt][\dx,\dxt]| \\
        \le \hat\B R \sqrt{\nDx^2 + \nDxt^2} \sqrt{\ndx^2 + \ndxt^2}.
    \end{multline}
\end{lemma}

\begin{proof}
    With \eqref{eq:Delta-f1-bound} and \eqref{eq:Delta-f2-bound} we obtain
    \renewcommand{\myindent}{\hspace{2cm}}
    \begin{align*}
        &|(f''(\xi_2,\xi_{\tau,2})
        -f''(\xi_1,\xi_{\tau,1}))[\Dx,\Dxt][\dx,\dxt]| \\
        &\myindent\le
        |(f_1''(\xi_2,\xi_{\tau,2})
        -f_1''(\xi_1,\xi_{\tau,1}))[\Dx,\Dxt][\dx,\dxt]| \\
        &\myindent\quad+ |(f_2''(\xi_2,\xi_{\tau,2})
        -f_2''(\xi_1,\xi_{\tau,1}))[\Dx,\Dxt][\dx,\dxt]| \\
        &\myindent\le \hat\B R \sqrt{\nDx^2 + \nDxt^2} \sqrt{\ndx^2 + \ndxt^2}
    \end{align*}
    with $\hat\B = \hat\beta_1 + \hat\beta_2$.
\end{proof}

\end{document}